\documentclass{amsart}
\usepackage{amsmath,amscd,xypic,amssymb,combelow,color,enumitem,graphicx,float,comment}

\newtheorem{theorem}[subsection]{Theorem}  

\newtheorem{proposition}[subsection]{Proposition}

\newtheorem{definition}[subsection]{Definition}
\newtheorem{claim}[subsection]{Claim}

\newtheorem{remark}[subsection]{Remark}

\def\fb{{\mathfrak{b}}}
\def\fg{{\mathfrak{g}}}

\def\fsl{{\mathfrak{sl}}}
\def\fgl{{\mathfrak{gl}}}

\def\hgl{{\widehat{\fgl}}}

\def\BC{{\mathbb{C}}}

\def\BN{{\mathbb{N}}}

\def\BP{{\mathbb{P}}}
\def\BR{{\mathbb{R}}}

\def\BZ{{\mathbb{Z}}}

\def\woo{\widehat{\otimes}}

\def\CA{{\mathcal{A}}}
\def\CB{{\mathcal{B}}}

\def\CO{{\mathcal{O}}}

\def\CS{{\mathcal{S}}}

\def\CV{{\mathcal{V}}}

\def\ph{\varphi}

\def\sym{\textrm{Sym}}

\def\UU{U_q(L\fg)}
\def\UUp{U_q^+(L\fg)}

\def\UUm{U_q^-(L\fg)}

\def\tUU{\widetilde{U}_q(L\fg)}

\def\bd{{\mathbf{d}}}

\def\br{{\mathbf{r}}}

\def\bs{\boldsymbol{\varsigma}}

\def\cc{{\mathbb{C}^I}}

\def\nn{{\mathbb{N}^I}}
\def\zz{{\mathbb{Z}^I}}
\def\rr{{\mathbb{R}^I}}

\def\balpha{{\boldsymbol{\alpha}}}

\def\bom{{\boldsymbol{\la}}}

\def\bpsi{{\boldsymbol{\psi}}}

\def\bom{{\boldsymbol{\omega}}}

\def\bm{{\boldsymbol{m}}}
\def\bn{{\boldsymbol{n}}}
\def\bp{{\boldsymbol{p}}}

\def\bx{{\boldsymbol{x}}}
\def\by{{\boldsymbol{y}}}

\def\b0{{\boldsymbol{0}}}
\def\bone{{\boldsymbol{1}}}

\def\hdeg{\text{hdeg }}
\def\vdeg{\text{vdeg }}

\def\UUaff{U_q(\widehat{\fg})_{c=1}}

\def\UUaffg{U_q(\widehat{\fb}^+)_{c=1}}
\def\UUaffl{U_q(\widehat{\fb}^-)_{c=1}}

\def\bx{\boldsymbol{x}}

\def\bone{{\boldsymbol{1}}}

\def\oCS{\mathring{\CS}}

\def\ord{\textbf{ord }}
\def\eord{\textbf{\emph{ord }}}

\def\binfty{\boldsymbol{\infty}}

\begin{document}

\title[Characters of quantum loop algebras]{\Large{\textbf{Characters of quantum loop algebras}}} 

\author[Andrei Negu\cb t]{Andrei Negu\cb t}

\address{École Polytechnique Fédérale de Lausanne (EPFL), Lausanne, Switzerland \newline \text{ } \ \ Simion Stoilow Institute of Mathematics (IMAR), Bucharest, Romania} 

\email{andrei.negut@gmail.com}

\maketitle
	
\begin{abstract} The $q$-characters of quantum loop algebras are very important objects in representation theory. In \cite{N Cat}, we showed that $q$-characters factor as a power series of the form studied in \cite{FM} times a character, an important phenomenon which had already been known in finite types. In the present paper, we prove a conjectural formula for the aforementioned character factor. \end{abstract}

\bigskip

\section{Introduction}
\label{sec:intro}

\medskip

\subsection{$q$-characters}

Fix a non-zero complex number $q$, which is not a root of unity. Let $\fg$ be a semisimple finite-dimensional complex Lie algebra, for which we fix a set of simple roots $\{\bs^i\}_{i \in I}$ and the corresponding set of positive roots $\Delta^+$. The representation theory of the quantum loop algebra
$$
\UU = \BC \Big \langle e_{i,d}, f_{i,d}, \ph_{i,d'}^\pm \Big \rangle_{i \in I, d \in \BZ, d' \geq 0} \Big/ \text{relations}
$$
has been long studied, see for instance \cite{CP, HJ}. Specifically, the simple representations $L(\bpsi)$ of (the Borel subalgebra of) $\UU$ are indexed by so-called $\ell$-weights
$$
\bpsi = \left(\psi_i(z) = \sum_{d=0}^{\infty} \frac {\psi_{i,d}}{z^d} \right)_{i \in I}
$$
which are rational, i.e. each $\psi_i(z)$ is the power series expansion of a rational function. Moreover, $\ell$-weights feature in the $q$-characters (\cite{FR}) of representations $V$ of (the Borel subalgebra of) $\UU$, which are defined as
$$
\chi_q(V) = \sum_{\ell\text{-weights } \bpsi} \dim_{\BC} (V_{\bpsi}) [\bpsi]
$$
Above, $V_\bpsi$ denotes the generalized eigenspace of $V$ in which the commuting family of operators $\{\ph_{i,d}^+\}_{i\in I, d \geq 0}$ acts according to the eigenvalues $\{\psi_{i,d}\}_{i\in I, d\geq 0}$, and $[\bpsi]$ denotes certain formal symbols that are required to satisfy $[\bpsi][\bpsi'] = [\bpsi\bpsi']$ with respect to component-wise multiplication of $I$-tuples of rational functions.

\medskip

\subsection{Characters}

It has long been observed (this was communicated to the author by David Hernandez, to whom we are very grateful, and it also features in \cite{W}) that the $q$-characters of simple representations $L(\bpsi)$ factor as
\begin{equation}
\label{eqn:formula observed}
\chi_q(L(\bpsi)) = \chi_q^{\neq 0}(L(\bpsi)) \cdot \chi^{\ord \bpsi} 
\end{equation}
where we write $\ord \bpsi = (\text{order of the pole at }0 \text{ of }\psi_i(z))_{i\in I} \in \zz$, and

\medskip

\begin{itemize}[leftmargin=*]

\item $\chi_q^{\neq 0}(L(\bpsi))$ is $[\bpsi]$ times a series in the symbols $\{A_{i,x}^{-1}\}_{i \in I, x \in \BC^*}$ of \cite{FR}, see \eqref{eqn:fm}.

\medskip

\item $\chi^{\ord \bpsi}$ is a character, i.e. a power series in the formal symbols $q^{-\bn}$ for various $\bn = (n_i)_{i\in I}$ lying in the non-negative orthant of the root lattice $\nn$; we identify $q^{-\bn}$ with the $I$-tuples of constant power series $\psi_i(z) = q^{-(\bn,\bs^i)}$.

\end{itemize}

\medskip

\noindent As far as identifying the factors that appear in \eqref{eqn:formula observed} is concerned, the left-most term is conjectured to coincide with the $q$-character of a representation of an appropriately shifted quantum loop algebra (\cite{H Shifted}). As for the right-most term, it is conjectured (\cite{W}, see also \cite{MY} for the case $\br \in \nn$) to be given by the following formula
\begin{equation}
\label{eqn:main conj}
\chi^{\br} = \prod_{\balpha \in \Delta^+} \left( \frac 1{1-q^{-\balpha}} \right)^{\max(0,\br \cdot \balpha)}, \quad \forall \br \in \zz
\end{equation}
where $\br \cdot \bn = \sum_{i \in I} r_i n_i$. In the present paper, we will prove a refined version of formula \eqref{eqn:main conj}, as follows: we realize $\chi^{\br}$ as the graded dimension of the vector space
\begin{equation}
\label{eqn:vector space}
L^{\br} = \CS_{<\b0}^- \Big/ \left( \text{kernel of the pairing }  \CS_{\geq -\br}^+ \otimes \CS_{<\b0}^- \xrightarrow{\langle \cdot, S(\cdot) \rangle} \BC \right)
\end{equation}
(see Subsections \ref{sub:pairing}, \ref{sub:slope}, \ref{sub:zero}, \ref{sub:characters} for the notation above) and observe that the vector space $L^\br$ has an additional grading by $d \in \{0,1,2,\dots\}$ (which matches the interesting grading on prefundamental modules discovered by \cite{FH}). Then one may ask to compute the refined $q$-character $\chi_{\text{ref}}^{\br}$ of the vector space \eqref{eqn:vector space}, which depends on an additional variable $v$ that keeps track of the grading by $d$.

\medskip

\begin{theorem}
	\label{thm:main}

For any semisimple complex Lie algebra $\fg$, we have
\begin{equation}
\label{eqn:main}
\chi_{\emph{ref}}^{\br} = \prod_{\balpha \in \Delta^+} \prod_{d=1}^{\max(0,\br \cdot \balpha)}  \frac 1{1-q^{-\balpha}v^d} 
\end{equation}
for all $\br \in \zz$. 

\end{theorem}

\medskip

\subsection{Kac-Moody Lie algebras} In \cite{N Cat}, we extended the category $\CO$ of representations of (the Borel subalgebra of) $\UU$ to arbitrary Kac-Moody Lie algebras $\fg$. Our main tool for doing so was the shuffle algebra, which we will recall in Section \ref{sec:shuffle}. In a nutshell, the double shuffle algebra $\CS$ of \eqref{eqn:double shuffle} features an isomorphism
$$
\CS \cong \UU
$$
When $\fg$ is of finite type, the above isomorphism sends the subalgebras $\CS_{<\b0}^-, \CS_{\geq \b0}^+$ of the LHS (defined as in Subsection \ref{sub:slope}) onto halves of the Borel subalgebra of the RHS  (see Proposition \ref{prop:slope affine} for the precise statement).  In \cite{N Cat}, we used the shuffle algebra language to provide an explicit construction of the simple modules $L(\bpsi)$ and of their $q$-characters. In the present paper, we will refine the notion of slopes in shuffle algebras from \cite{N R-matrix} in order to prove Theorem \ref{thm:main}. In fact, the result that we will prove in Subsection \ref{sub:characters} actually applies to any Kac-Moody Lie algebra $\fg$, in that we show 
\begin{equation}
\label{eqn:thm intro}
\chi_{\text{ref}}^{\br} = \prod_{\bn \in \nn \backslash \b0} \prod_{d=1}^{\max(0,\br \cdot \bn)} \left( \frac 1{1-q^{-\bn}v^d} \right)^{a_{\fg,\bn}}
\end{equation}
where the numbers $\{a_{\fg,\bn}\}_{\bn \in \nn}$ are defined in \eqref{eqn:conj hor}. However, we only prove \eqref{eqn:thm intro} conditional on the conjectural formula \eqref{eqn:conj}. The aforementioned conjecture is known when $\fg$ is of finite type, which establishes \eqref{eqn:main}, but this conjecture is also known in related cases such as that of quantum toroidal $\fgl_1$ studied in \cite{FJMM} (which corresponds to $\fg = \hgl_1$). 

\medskip

\subsection{Acknowledgements} I would like to thank David Hernandez for his great help and patience explaining the many aspects of the theory of $q$-characters and category $\CO$ for quantum loop algebras, and for his initial suggestion to pursue the computation of the characters $\chi^{\br}$.

\bigskip

\section{Shuffle algebras}
\label{sec:shuffle}

\medskip

\subsection{Basic definitions}
\label{sub:basic}

The set $\BN$ is assumed to contain 0 throughout the present paper. Fix a finite set $I$ and a symmetrizable Cartan matrix
\begin{equation}
	\label{eqn:cartan matrix}
	\left(\frac {2d_{ij}}{d_{ii}} \in \BZ\right)_{i,j \in I}
\end{equation}
where $d_{ii}$ are positive integers with greatest common divisor 2, and $d_{ij} = d_{ji}$ are non-positive integers. This data corresponds to a Kac-Moody Lie algebra $\fg$, which we will refer to as the ``type" of all algebras considered hereafter. $I$ is interpreted as a set of simple roots of $\fg$, and $\zz$ should be interpreted as the root lattice. The Cartan matrix \eqref{eqn:cartan matrix} is called symmetric, or equivalently $\fg$ is called simply laced, if
\begin{equation}
	\label{eqn:symmetric matrix}
	d_{ii} = 2, \quad \ \forall i \in I
\end{equation}
Knowledge of the integers $\{d_{ij}\}_{i,j \in I}$ provides a symmetric bilinear pairing 
\begin{equation}
	\label{eqn:symmetric pairing}
	\cc \otimes \cc \xrightarrow{( \cdot,\cdot )} \BC, \qquad (\bs^i, \bs^j) = d_{ij}
\end{equation}
where $\bs^i = \underbrace{(0,\dots,0,1,0,\dots,0)}_{1\text{ on the }i\text{-th position}}$. For any $\bm = (m_i)_{i \in I}, \bn = (n_i)_{i \in I} \in \cc$, we let
\begin{equation}
	\label{eqn:dot product}
	\bm \cdot \bn = \sum_{i \in I} m_i n_i
\end{equation}
For $\bm,\bn \in \rr$, we will write $\bm \leq \bn$ for the partial order determined by $m_i \leq n_i$ for all $i \in I$. We will also use the notation $\b0 = (0,\dots,0)$ and $\bone = (1,\dots,1)$.

\medskip

\subsection{The quantum loop algebra}

We briefly recall the definition of the quantum loop algebra associated to the Kac-Moody Lie algebra $\fg$ (see \cite{Dr} for the finite type case, \cite{N Loop} for simply-laced Kac-Moody types, and \cite{N Arbitrary} for expectations and conjectures on the general case). Consider the following formal series for all $i \in I$
\begin{equation}
\label{eqn:formal series}
  e_i(x) = \sum_{d \in \BZ} \frac {e_{i,d}}{x^d}, \qquad
  f_i(x) = \sum_{d \in \BZ} \frac {f_{i,d}}{x^d}, \qquad
  \ph^\pm_i(x) = \sum_{d = 0}^\infty \frac {\ph^\pm_{i,d}}{x^{\pm d}}
\end{equation}
and let $\delta(x) = \sum_{d \in \BZ} x^d$ be the formal delta function. For any $i,j \in I$, set
\begin{equation}
\label{eqn:zeta}
  \zeta_{ij} (x) = \frac {x - q^{-d_{ij}}}{x - 1}
\end{equation}
For any $i \in I$, we will write $q_i = q^{\frac {d_{ii}}2}$. 

\medskip

\begin{definition}
\label{def:pre quantum loop}

The pre-quantum loop algebra associated to $\fg$ is
\begin{equation}
\label{eqn:pre-quantum}
\tUU = \BC \Big \langle e_{i,d}, f_{i,d}, \ph_{i,d'}^\pm \Big \rangle_{i \in I, d \in \BZ, d' \geq 0} \Big/ \text{relations \eqref{eqn:rel 0 loop}-\eqref{eqn:rel 3 loop}}
\end{equation}
where we impose the following relations for all $i,j \in I$ and all $\pm, \pm' \in \{+,-\}$:
\begin{equation}
\label{eqn:rel 0 loop}
e_i(x) e_j(y) \zeta_{ji} \left( \frac yx \right) =\, e_j(y) e_i(x) \zeta_{ij} \left(\frac xy \right)
\end{equation}
\begin{equation}
\label{eqn:rel 1 loop}
  \ph_j^\pm(y) e_i(x) \zeta_{ij} \left(\frac xy \right) = e_i(x) \ph_j^\pm(y) \zeta_{ji} \left( \frac yx \right)
\end{equation}
\begin{equation}
\label{eqn:rel 2 loop}
  \ph_{i}^{\pm}(x) \ph_{j}^{\pm'}(y) = \ph_{j}^{\pm'}(y) \ph_{i}^{\pm}(x), \quad
  \ph_{i,0}^+ \ph_{i,0}^- = 1
\end{equation}
as well as the opposite relations with $e$'s replaced by $f$'s, and finally the relation
\begin{equation}
\label{eqn:rel 3 loop}
  \left[ e_i(x), f_j(y) \right] =
  \frac {\delta_{ij}\delta \left(\frac xy \right)}{q_i-q_i^{-1}}  \Big( \ph_i^+(x) - \ph_i^-(y) \Big)
\end{equation}
We define the quantum loop algebra associated to $\fg$ as
\begin{equation}
\label{eqn:quantum}
\UU = \tUU \Big / (\text{kernel of the Hopf pairing})
\end{equation}
(see \cite[Subsection 2.5]{N Cat} for the meaning of the right-hand side of \eqref{eqn:quantum}, and for an explicit description of the aforementioned kernel in finite and simply laced types).

\end{definition}

\medskip

\subsection{The big shuffle algebra}
\label{sub:big shuffle}

We now review the trigonometric version (\cite{E1, E2}) of the Feigin-Odesskii shuffle algebra (\cite{FO}) associated to a Kac-Moody Lie algebra $\fg$. Consider the vector space of rational functions in arbitrarily many variables
\begin{equation}
	\label{eqn:big shuffle}
	\CV = \bigoplus_{\bn \in \nn} \CV_{\bn}, \quad \text{where} \quad \CV_{(n_i)_{i \in I}} = \frac {\BC[z_{i1}^{\pm 1},\dots,z_{in_i}^{\pm 1}]^{\text{sym}}_{i \in I}}{\prod^{\text{unordered}}_{\{i \neq j\} \subset I} \prod_{1\leq a \leq n_i, 1\leq b \leq n_j} (z_{ia} - z_{jb})}
\end{equation}
Above, ``sym" refers to \textbf{color-symmetric} Laurent polynomials, meaning that they are symmetric in the variables $z_{i1},\dots,z_{in_i}$ for each $i \in I$ separately (the terminology is inspired by the fact that $i \in I$ is called the color of the variable $z_{ia}$). We make the vector space $\CV$ into a $\BC$-algebra via the following shuffle product:
\begin{equation}
	\label{eqn:mult}
	E( z_{i1}, \dots, z_{i n_i}) * E'(z_{i1}, \dots,z_{i n'_i}) = \frac 1{\bn!  \bn'!}\,\cdot
\end{equation}
$$
\textrm{Sym} \left[ E(z_{i1}, \dots, z_{in_i}) E'(z_{i,n_i+1}, \dots, z_{i,n_i+n'_i})
\prod_{i,j \in I} \mathop{\prod_{1 \leq a \leq n_i}}_{n_j < b \leq n_j+n_j'} \zeta_{ij} \left( \frac {z_{ia}}{z_{jb}} \right) \right]
$$
The word $\sym$ denotes symmetrization with respect to the
\begin{equation*}
	(\bn+\bn')! := \prod_{i\in I} (n_i+n'_i)!
\end{equation*}
permutations of the variables $\{z_{i1}, \dots, z_{i,n_i+n'_i}\}$ for each $i$ independently. The reason for this definition is that there exist algebra homomorphisms
\begin{align}
&\Upsilon^+ : \UUp \rightarrow \CV, \quad \ \ e_{i,d} \mapsto z_{i1}^d \in \CV_{\bs^i} \label{eqn:upsilon plus} \\
&\Upsilon^- : \UUm \rightarrow \CV^{\text{op}}, \quad f_{i,d} \mapsto z_{i1}^d \in \CV^{\text{op}}_{\bs^i} \label{eqn:upsilon minus}
\end{align}
where $\UUp, \UUm$ denote the subalgebras of $\UU$ generated by $\{e_{i,d}\}_{i \in I, d\in \BZ}$ and $\{f_{i,d}\}_{i \in I, d\in \BZ}$, respectively. It was shown in \cite{N Arbitrary} that the homomorphisms above are injective, i.e. the kernel of the Hopf pairing by which we factor in \eqref{eqn:quantum} precisely contains all possible relations between $e_{i,d}$ and $f_{i,d}$ inside $\CV$ and $\CV^{\text{op}}$, respectively. 

\medskip

\subsection{The shuffle algebra}
\label{sub:small shuffle}

Due to the injective homomorphisms $\Upsilon^\pm$ above, we will often abuse notation by denoting $e_{i,d} = z_{i1}^d \in \CV_{\bs^i}$ and $f_{i,d} = z_{i1}^d \in \CV_{\bs^i}^{\text{op}}$. Then
\begin{equation}
\label{eqn:spherical}
\CS^\pm = \text{Im }\Upsilon^\pm
\end{equation}
are called \textbf{shuffle algebras}, i.e. the subalgebras of $\CV$ and $\CV^{\text{op}}$ (when $\pm$  is $+$ and $-$, respectively) generated by $\{e_{i,d}\}_{i \in I, d\in \BZ}$ and $\{f_{i,d}\}_{i \in I, d\in \BZ}$ (respectively). 

\medskip

\begin{remark} 

For $\fg$ of finite type, we have the following explicit description
\begin{equation}
	\label{eqn:e}
	\CS^\pm = \left\{  \frac {\rho(z_{i1},\dots,z_{in_i})_{i \in I}}
	{\prod^{\text{unordered}}_{\{i \neq j\} \subset I} \prod_{1\leq a \leq n_i, 1\leq b \leq n_{j}} (z_{ia} - z_{jb})} \right\}
\end{equation}
where $\rho$ goes over the set of color-symmetric Laurent polynomials that satisfy the Feigin-Odesskii \textbf{wheel conditions}:
\begin{equation}
	\label{eqn:wheel}
	\rho(\dots, z_{ia}, \dots,z_{jb},\dots)\Big|_{(z_{i1},z_{i2}, \dots, z_{in}) \mapsto (w, w q^{d_{ii}},  \dots, w q^{(n-1)d_{ii}}),\, z_{j1} \mapsto w q^{-d_{ij}}} =  0
\end{equation}
for any $i \neq j$ in $I$, where $n = 1 - \frac {2d_{ij}}{d_{ii}}$. The inclusion $\subseteq$ of \eqref{eqn:e} was established in \cite{E1, E2} based on the seminal work of \cite{FO}, while the inclusion $\supseteq$ was proved in \cite{NT}.

\end{remark}

\medskip 

\noindent Elements of either $\CS^+$ or $\CS^-$ will be referred to as \textbf{shuffle elements}. We may grade the shuffle algebras $\CS^\pm$ by $(\pm \nn) \times \BZ$ via
\begin{equation}
	\label{eqn:deg}
	\deg X = (\pm \bn, d)
\end{equation}
for any $X(z_{i1},\dots,z_{in_i}) \in \CS^\pm$ of total homogeneous degree $d$, where $\bn = (n_i)_{i \in I}$. We will call $\pm \bn$ and $d$ the \textbf{horizontal} and \textbf{vertical} degrees of $X$, and denote them
\begin{equation}
	\label{eqn:hdeg vdeg}
	\hdeg X = \pm \bn \quad \text{and} \quad \vdeg X = d
\end{equation}
In terms of this grading, shuffle algebras decompose as
\begin{equation}
	\label{eqn:graded pieces}
	\CS^\pm = \bigoplus_{\bn \in \nn} \CS_{\pm \bn} = \bigoplus_{\bn \in \nn} \bigoplus_{d \in \BZ} \CS_{\pm \bn, d}
\end{equation}

\medskip

\begin{proposition}
\label{prop:shift}
	
	For any $\br \in \zz$, we have \textbf{shift} automorphisms of $\CS^\pm$
	\begin{equation}
		\label{eqn:shift}
		\CS_{\pm \bn,\bd} \xrightarrow{\sigma_{\br}} \CS_{\pm \bn,\bd \pm \br\cdot \bn}, \qquad R(z_{i1},\dots,z_{in_i}) \mapsto R(z_{i1},\dots,z_{in_i}) \prod_{i \in I} \prod_{a=1}^{n_i} z_{ia}^{\pm r_i}
	\end{equation}	
and note that they satisfy $\sigma_{\br} \circ \sigma_{\br'} = \sigma_{\br+\br'} = \sigma_{\br'} \circ \sigma_{\br}$.
	
\end{proposition}

\medskip

\subsection{The double shuffle algebra}
\label{sub:double shuffle}

Define 
\begin{equation}
\label{eqn:double shuffle}
\CS = \CS^+ \otimes \frac {\BC [\ph_{i,d}^\pm]_{i \in I, d \geq 0}}{\ph_{i,0}^+ \ph_{i,0}^- - 1} \otimes \CS^-
\end{equation}
endowed with the following commutation relations for all $E \in \CS^+$, $F \in \CS^-$, $j \in I$
\begin{align}
	&\ph_j^\pm(y) E(z_{i1},\dots,z_{in_i}) = E(z_{i1},\dots,z_{in_i}) \ph_j^\pm(y) \prod_{i \in I} \prod_{a=1}^{n_i}\frac {\zeta_{ji} \left(\frac y{z_{ia}} \right)}{\zeta_{ij} \left(\frac {z_{ia}}y \right)} \label{eqn:shuffle plus commute} \\
	&F(z_{i1},\dots,z_{in_i}) \ph_j^\pm(y)  = \ph_j^\pm(y) F(z_{i1},\dots,z_{in_i}) \prod_{i \in I} \prod_{a=1}^{n_i}\frac {\zeta_{ji} \left(\frac y{z_{ia}} \right)}{\zeta_{ij} \left(\frac {z_{ia}}y \right)} \label{eqn:shuffle minus commute}
\end{align}
(the right-hand sides of \eqref{eqn:shuffle plus commute}-\eqref{eqn:shuffle minus commute} are expanded in non-positive powers of $y^{\pm 1}$) and
\begin{equation}
	\label{eqn:shuffle plus minus commute}
	\left[ e_{i,d}, f_{j,d'} \right] = \frac {\delta_{ij}}{q_i-q_i^{-1}}  \cdot \begin{cases} \ph_{i,d+d'}^+ &\text{if }d+d' > 0 \\ \ph_{i,0}^+ - \ph_{i,0}^- &\text{if }d+d' = 0 \\ - \ph_{i,-d-d'}^- &\text{if }d+d' < 0 \end{cases}
\end{equation}
Because $\CS^+$ and $\CS^-$ are by definition generated by $e_{i,d}$ and $f_{j,d'}$ as $i,j \in I, d,d' \in \BZ$, the relations above are sufficient to completely reorder any product of elements coming from the three tensor factors in \eqref{eqn:double shuffle}. These specific relations are chosen so that the homomorphisms $\Upsilon^\pm$ of \eqref{eqn:upsilon plus}-\eqref{eqn:upsilon minus} glue to an isomorphism
\begin{equation}
\label{eqn:upsilon}
\Upsilon : \UU \xrightarrow{\sim} \CS
\end{equation}
It will be useful to replace the generators $\ph^\pm_j(y)$ by $\{\kappa_j, p_{j,u}\}_{j \in I, u \in \BZ\backslash 0}$ defined by
$$
\ph^\pm_j(y) = \kappa_j^{\pm 1} \exp \left( \sum_{u=1}^{\infty} \frac {p_{j,\pm u}}{u y^{\pm u}} \right)
$$
In terms of these new generators, relations \eqref{eqn:shuffle plus commute} and \eqref{eqn:shuffle minus commute} take the form
\begin{equation}
	\label{eqn:k shuffle}
	\kappa_j X = X \kappa_j q^{(\pm \bn, \bs^j)}
\end{equation}
\begin{equation}
	\label{eqn:p shuffle}
	\left[p_{j,u}, X \right] = \pm X \sum_{i \in I} \left(z_{i1}^u+\dots+z_{in_i}^u\right)(q^{ud_{ij}} - q^{-ud_{ij}})
\end{equation}
for any $X(z_{i1},\dots,z_{in_i})_{i \in I} \in \CS_{\pm \bn}$.

\medskip

\subsection{The Hopf algebra structure}
\label{sub:shuffle Hopf}

There exist topological coproducts on
\begin{align}
	\CS^{\geq} = \CS^+ \otimes \BC [\ph_{i,d}^+]_{i \in I, d \geq 0} \label{eqn:s geq} \\
	\CS^{\leq} = \BC [\ph_{i,d}^-]_{i \in I, d \geq 0} \otimes \CS^- \label{eqn:s leq}
\end{align}
(we tacitly assume that $\ph_{i,0}^\pm$ are replaced by $\kappa_i^{\pm 1}$ in the formulas above, so they are assumed to be invertible) defined by 
\begin{equation}
	\label{eqn:coproduct ph}
	\Delta(\ph^\pm_i(z)) = \ph^\pm_i(z) \otimes \ph^\pm_i(z)
\end{equation}
and
\begin{align}
	\Delta(E) = \sum_{\b0 \leq \bm \leq \bn} \frac {\prod^{j \in I}_{m_j < b \leq n_j} \ph^+_j(z_{jb}) E(z_{i1},\dots , z_{im_i} \otimes z_{i,m_i+1}, \dots, z_{in_i})}{\prod^{i \in I}_{1\leq a \leq m_i} \prod^{j \in I}_{m_j < b \leq n_j} \zeta_{ji} \left( \frac {z_{jb}}{z_{ia}} \right)} \label{eqn:coproduct shuffle plus} \\
	\Delta(F) = \sum_{\b0 \leq \bm \leq \bn} \frac {F(z_{i1},\dots , z_{im_i} \otimes z_{i,m_i+1}, \dots, z_{in_i}) \prod^{j \in I}_{1 \leq b \leq m_j} \ph^-_j(z_{jb})}{\prod^{i \in I}_{1\leq a \leq m_i} \prod^{j \in I}_{m_j < b \leq n_j} \zeta_{ij} \left( \frac {z_{ia}}{z_{jb}} \right)} \label{eqn:coproduct shuffle minus} 
\end{align}
for all $E \in \CS_{\bn}$, $F \in \CS_{-\bn}$. To make sense of the right-hand side of \eqref{eqn:coproduct shuffle plus} and \eqref{eqn:coproduct shuffle minus}, we expand the denominator as a power series in the range $|z_{ia}| \ll |z_{jb}|$, and place all the powers of $z_{ia}$ to the left of the $\otimes$ sign and all the powers of $z_{jb}$ to the right of the $\otimes$ sign (for all $i,j \in I$, $1 \leq a \leq m_i$, $m_j < b \leq n_j$). There also exists a topological antipode on $\CS^\geq$ and $\CS^\leq$, given by
\begin{equation}
\label{sub:antipode ph}
S \left( \ph^\pm_i(z) \right) = \ph^\pm_i(z)^{-1} := \sum_{d=0}^\infty \frac {\bar{\ph}_{i,d}^\pm}{z^{\pm d}} 
\end{equation}
\begin{equation}
\label{eqn:antipode e and f}
S(e_{i,d}) = - \sum_{u=0}^\infty \bar{\ph}_{i,u}^+ e_{i,d-u}  \quad \text{and} \quad S(f_{i,d}) = - \sum_{u=0}^\infty f_{i,d+u}\bar{\ph}_{i,u}^-
\end{equation}
(these formulas determine the antipode completely, as $\CS^+$ and $\CS^-$ are generated by $e_{i,d}$ and $f_{i,d}$, respectively). The inverse antipode $S^{-1}$ is given by the same formulas as \eqref{eqn:antipode e and f}, but with $\bar{\ph}$ placed on the opposite side of $e$ and $f$.

\medskip 

\subsection{The pairing}
\label{sub:pairing}

We will call $i_1,\dots,i_n \in I$ an \textbf{ordering} of $\bn \in \nn$ if
$$
\bs^{i_1} + \dots + \bs^{i_n} = \bn
$$
If this happens, we will employ for any $X(z_{i1},\dots,z_{in_i})_{i \in I} \in \CS_{\pm \bn}$ the notation
\begin{equation}
\label{eqn:notation ordering}
X(z_1,\dots,z_n)
\end{equation} 
to indicate the fact that each symbol $z_a$ is plugged into a variable of the form $z_{i_a \bullet_a}$ of $X$, where the choice of $\bullet_a \geq 1$ does not matter due to the color-symmetry of $X$.
	
\medskip

\noindent Following \cite{N Wheel, NT}, there exists a non-degenerate pairing
\begin{equation}
	\label{eqn:pairing shuffle}
	\CS^+ \otimes \CS^- \xrightarrow{\langle \cdot, \cdot \rangle} \BC
\end{equation}
completely determined by
\begin{equation}
\Big \langle e_{i_1,d_1} * \dots * e_{i_n,d_n}, F \Big \rangle = \int_{|z_1| \gg \dots \gg |z_n|} \frac {z_1^{d_1} \dots z_n^{d_n} F(z_1,\dots,z_n)}{\prod_{1\leq a < b \leq n} \zeta_{i_bi_a} \left(\frac {z_b}{z_a} \right)} \label{eqn:pairing shuffle minus} 
\end{equation}
for all $F \in \CS_{-\bn}$, any $d_1,\dots,d_n \in \BZ$ and any ordering $i_1,\dots,i_n$ of $\bn$ (the notation in the RHS of \eqref{eqn:pairing shuffle minus} is defined as in \eqref{eqn:notation ordering}). Throughout the present paper, the notation 
$$
\int_{|z_1| \gg  \dots \gg |z_n|} R(z_1,\dots,z_n) \quad \text{is shorthand for} \ \int_{|z_1| \gg \dots \gg |z_n|} R(z_1,\dots,z_n)  \prod_{a=1}^n \frac {dz_a}{2 \pi i z_a}
$$
and refers to the contour integral over concentric circles centered at the origin of the complex plane (the notation $|z_a| \gg |z_b|$ means that the these circles are very far away from each other when compared to any constants that might appear in the formula for $R(z_1,\dots,z_n)$). The following was proved in \cite[Lemma 3.10]{N Cat}
\begin{equation}
\Big \langle e_{i_1,d_1} * \dots * e_{i_n,d_n}, S(F) \Big\rangle = (-1)^n \int_{|z_1| \ll \dots \ll |z_n|} \frac {z_1^{d_1} \dots z_n^{d_n} F(z_1,\dots,z_n)}{\prod_{1\leq a < b \leq n} \zeta_{i_bi_a} \left(\frac {z_b}{z_a} \right)} \label{eqn:antipode pairing shuffle}
\end{equation}
The pairing \eqref{eqn:pairing shuffle} extends to a Hopf pairing
\begin{equation}
	\label{eqn:enhanced pairing}
\CS^\geq \otimes \CS^\leq \xrightarrow{\langle \cdot, \cdot \rangle} \BC 
\end{equation}
by setting $ \langle \ph^+_i(x), \ph^-_j(y)  \rangle =  \frac {\zeta_{ij} \left(\frac xy \right)}{\zeta_{ji} \left(\frac yx \right)} $ (expanded as $|x| \gg |y|$), and requiring
\begin{align}
	&\Big \langle a,b_1b_2 \Big \rangle = \Big \langle \Delta(a), b_1 \otimes b_2 \Big \rangle \label{eqn:bialgebra 1} \\
	&\Big \langle a_1 a_2 ,b \Big \rangle = \Big \langle a_1 \otimes a_2, \Delta^{\text{op}}(b) \Big \rangle \label{eqn:bialgebra 2} 
\end{align}
($\Delta^{\text{op}}$ is the opposite coproduct) and
\begin{equation}
	\label{eqn:antipode pairing}
	\Big \langle S(a), S(b) \Big \rangle = \Big \langle a,b \Big \rangle
\end{equation}
for all $a,a_1,a_2 \in \CS^\geq$ and $b,b_1,b_2 \in \CS^\leq$. With this in mind, $\CS$ of \eqref{eqn:double shuffle} becomes isomorphic to the Drinfeld double $\CS^\geq \otimes \CS^\leq$ (modulo the identification of the element $\kappa_i \otimes 1$ with the inverse of the element $1 \otimes \kappa_i^{-1}$, for all $i \in I$), meaning that we have
\begin{equation}
	\label{eqn:drinfeld double}
	ba = \Big \langle a_1, S(b_1) \Big \rangle a_2b_2 \Big \langle a_3, b_3 \Big \rangle
\end{equation}
for all $a \in \CS^\geq$ and $b \in \CS^\leq$ \footnote{Note a small imprecision: for \eqref{eqn:drinfeld double} to hold as stated, one would need to rescale the RHS of \eqref{eqn:pairing shuffle minus} by $\prod_{a=1}^n (q_{i_a}- q^{-1}_{i_a})^{-1}$, but we will ignore this so as to not overburden the notation.}. In the formula above, we use Sweedler notation 
$$
\Delta^{(2)}(a) = a_1 \otimes a_2 \otimes a_3, \quad \Delta^{(2)}(b) = b_1 \otimes b_2 \otimes b_3
$$
(where $\Delta^{(2)} = (\Delta \otimes \text{Id} ) \circ \Delta$) to avoid writing down the implied summation signs.
	
\medskip 

\begin{proposition}
\label{prop:shifts}

Shifts induce Hopf algebra automorphisms
\begin{equation}
\label{eqn:shift geq}
\sigma_{\br} : \CS^{\geq}  \rightarrow \CS^{\geq} \qquad \text{and} \qquad \sigma_{\br} : \CS^{\leq}  \rightarrow \CS^{\leq}
\end{equation}
that preserve the pairing
\begin{equation}
\label{eqn:shift pairing}
\langle \sigma_{\br}(a), \sigma_{\br}(b) \rangle = \langle a,b \rangle 
\end{equation}
for all $a \in \CS^{\geq}$, $b \in \CS^{\leq}$. Thus, shifts induce Hopf algebra automorphisms
\begin{equation}
\label{eqn:shift double}
\sigma_{\br} : \CS \rightarrow \CS
\end{equation}
of the Drinfeld double.

\end{proposition}	

\medskip 

\begin{proof} The automorphisms \eqref{eqn:shift geq} are obtained from \eqref{eqn:shift}  and 
$$
\sigma_{\br}(\ph^\pm_{i,d}) = \ph^{\pm}_{i,d}
$$
for all $i,d$. The fact that this gives rise to algebra automorphisms \eqref{eqn:shift geq} is because applying the shift to equations \eqref{eqn:shuffle plus commute}-\eqref{eqn:shuffle minus commute} amounts to both sides being multiplied by the same monomial. Similarly, \eqref{eqn:shift geq} is a coalgebra automorphism because applying the shift to equations \eqref{eqn:coproduct shuffle plus}-\eqref{eqn:coproduct shuffle minus} amounts to both sides being multiplied by the same monomial. The fact that the pairing \eqref{eqn:pairing shuffle} is preserved by shift is due to the fact that applying $\sigma_{\br}$ to the two arguments in LHS of \eqref{eqn:pairing shuffle minus} amounts to multiplying the integrand in RHS of \eqref{eqn:pairing shuffle minus} by $z_1^{r_{i_1}} \dots z_n^{r_{i_n}} z_1^{-r_{i_1}} \dots z_n^{-r_{i_n}} = 1$.
	
\end{proof} 
	
\medskip 	

\subsection{Slopes}
\label{sub:slope}

We will now define slope subalgebras in a higher generality than \cite{N Cat} (corresponding to any $\bp \in \rr$ instead of just $\bp = \mu \bone$), but the general theory continues to hold. The original source for these notions is \cite{N Shuffle, N R-matrix}, inspired by \cite{FHHSY}.

\medskip

\begin{definition}
	\label{def:slope}
	
For any $\bp \in \rr$, we will say that
	
	\medskip
	
	\begin{itemize}[leftmargin=*]
		
		\item $E \in \CS_\bn$ has \textbf{slope} $\geq \bp$ if the following limit is finite for all $\b0 < \bm \leq \bn$
		\begin{equation}
			\label{eqn:slope e geq}
			\lim_{\xi \rightarrow 0} \frac {E(\xi z_{i1},\dots,\xi z_{im_i},z_{i,m_{i+1}},\dots,z_{in_i})}{\xi^{\bp \cdot \bm}}
		\end{equation}

\medskip
		
\item $E \in \CS_\bn$ has \textbf{slope} $\leq \bp$ if the following limit is finite for all $\b0 < \bm \leq \bn$
\begin{equation}
	\label{eqn:slope e leq}
\lim_{\xi \rightarrow \infty} \frac {E(\xi z_{i1},\dots,\xi z_{im_i},z_{i,m_{i+1}},\dots,z_{in_i})}{\xi^{\bp \cdot \bm}}
\end{equation}
		
		\medskip
		
		\item $F \in \CS_{-\bn}$ has \textbf{slope} $\leq \bp$ if the following limit is finite for all $\b0 < \bm \leq \bn$
		\begin{equation}
			\label{eqn:slope f leq}
			\lim_{\xi \rightarrow 0} \frac {F(\xi z_{i1},\dots,\xi z_{im_i},z_{i,m_{i+1}},\dots,z_{in_i})}{\xi^{-\bp \cdot \bm}}
		\end{equation}
		
\medskip
		
\item $F \in \CS_{-\bn}$ has \textbf{slope} $\geq \bp$ if the following limit is finite for all $\b0 < \bm \leq \bn$
\begin{equation}
\label{eqn:slope f geq}
\lim_{\xi \rightarrow \infty} \frac {F(\xi z_{i1},\dots,\xi z_{im_i},z_{i,m_{i+1}},\dots,z_{in_i})}{\xi^{-\bp \cdot \bm}}
\end{equation}
		
	\end{itemize}

\end{definition}

\medskip

\noindent We will write $\CS^\pm_{\geq \bp}$ and $\CS^\pm_{\leq \bp}$ for the subalgebras of elements of $\CS^\pm$ of slope $\geq \bp$ and $\leq \bp$, respectively. We will denote their graded components by
\begin{align}
	&\CS^\pm_{\geq \bp}=\bigoplus_{\bn\in \nn} \CS_{\geq \bp|\pm \bn}=\bigoplus_{\bn\in \nn} \bigoplus_{d\in \BZ} \CS_{\geq \bp|\pm \bn,d} \label{eqn:shuffle slope geq} \\
	&\CS^\pm_{\leq \bp}=\bigoplus_{\bn\in \nn} \CS_{\leq \bp|\pm \bn}=\bigoplus_{\bn\in \nn} \bigoplus_{d\in \BZ} \CS_{\leq \bp|\pm \bn,d} \label{eqn:shuffle slope leq}
\end{align}
For any $\br \in \zz$, the following formulas are easy exercises
\begin{equation}
	\label{eqn:easy exercise}
\sigma_{\br} ( \CS^\pm_{\geq \bp} ) = \CS^\pm_{\geq \bp+\br} \qquad \text{and} \qquad \sigma_{\br} ( \CS^\pm_{\leq \bp} ) = \CS^\pm_{\leq \bp+\br} 
\end{equation}
If the limits in \eqref{eqn:slope e geq}-\eqref{eqn:slope f geq} are all 0, then we will say that $E$ and $F$ have slopes $<\bp$ or $>\bp$, respectively. All statements in the previous paragraph hold for the subalgebras $\CS^\pm_{> \bp}$ and $\CS^\pm_{< \bp}$ of elements of slope $> \bp$ and $< \bp$, respectively.

\medskip

\subsection{Slope subalgebras}
\label{sub:slope subalgebras}

If a shuffle element $X$ simultaneously has slope $\leq \bp$ and $\geq \bp$, then $\vdeg X = \bp \cdot \hdeg X$. The set of such elements will be denoted by
\begin{equation}
	\label{eqn:slope subalgebra}
	\CB_{\bp}^\pm = \mathop{\bigoplus_{(\bn,d) \in \nn \times \BZ}}_{d = \pm \bp \cdot \bn} \CB_{\bp|\pm \bn}
\end{equation}
and will be called a \textbf{slope subalgebra}. As a consequence of \eqref{eqn:easy exercise}, we have 
\begin{equation}
\label{eqn:sigma on b}
\sigma_{\br} ( \CB^\pm_{\bp} ) = \CB^\pm_{\bp+\br}
\end{equation}
for all $\bp \in \rr$ and $\br \in \zz$. For any fixed $\bp \in \rr$,
\begin{equation}
\label{eqn:enlarged slope}
\CB_{\bp}^\geq = \CB_{\bp}^+ \otimes \BC[\kappa_i]_{i \in I} \qquad \text{and} \qquad \CB_{\bp}^\leq = \BC[\kappa_i^{-1}]_{i \in I} \otimes \CB_{\bp}^-
\end{equation}
are subalgebras of $\CS^\geq$ and $\CS^\leq$, as long as we impose the commutation relation
\begin{equation}
	\label{eqn:k shuffle bis}
	\kappa_\bm X = X \kappa_\bm q^{(\hdeg X, \bm)}
\end{equation}
between the tensor factors, for all $X \in \CS^\pm $ and $\bm = (m_i)_{i \in I} \in \zz$, where we denote $\kappa_\bm = \prod_{i\in I} \kappa_i^{m_i}$. We can make \eqref{eqn:enlarged slope} into Hopf algebras, using the coproduct
\begin{equation}
\label{eqn:coproduct kappa}
\Delta_{\bp}(\kappa_i) = \kappa_i \otimes \kappa_i 
\end{equation}
for all $i \in I$, and
\begin{align}
	&\Delta_\bp(E) = \sum_{\b0 \leq \bm \leq \bn} (\kappa_{\bn-\bm} \otimes 1) (\text{value of the limit \eqref{eqn:slope e geq}}) \label{eqn:coproduct slope plus} \\
	&\Delta_\bp(F) = \sum_{\b0 \leq \bm \leq \bn} (\text{value of the limit \eqref{eqn:slope f leq}}) (1 \otimes \kappa_{\bm}^{-1}) \label{eqn:coproduct slope minus}
\end{align} 
$\forall E \in \CB^+_\bp, F \in \CB^-_\bp$. In the limits referenced in the formulas above, we place the $\otimes$ sign between the variables $\{z_{i1},\dots,z_{im_i}\}_{i \in I}$ and $\{z_{i,m_i+1},\dots,z_{in_i}\}_{i \in I}$ within any monomial. By analogy with \cite[Proposition V.3]{N Thesis}, we have the following.

\medskip

\begin{proposition}
	\label{prop:slope subalgebra}

For any $\bp \in \rr$, we may consider the Drinfeld double
\begin{equation}
\label{eqn:double slope}
\CB_\bp = \CB_\bp^\geq \otimes \CB_\bp^\leq 
\end{equation}
with respect to the Hopf algebra structure above and the restriction of \eqref{eqn:enhanced pairing} to 
\begin{equation}
\label{eqn:really restricted}
\CB_{\bp}^{\geq} \otimes \CB_{\bp}^{\leq} \rightarrow \BC
\end{equation}
Then the inclusion $\CB_\bp = \CB_\bp^\geq \otimes \CB_\bp^\leq  \subset \CS^\geq \otimes \CS^\leq = \CS$ is an algebra homomorphism, although not a bialgebra homomorphism.

\end{proposition}

\medskip

\subsection{Factorizations I}
\label{sub:factorizations}

Slope subalgebras are important because they are the building blocks of shuffle algebras, as we will now show. A parameterized curve $\bp(t) = (p_i(t))_{i \in I}$ in $\rr$ will be called \textbf{catty-corner} if
\begin{equation}
\label{eqn:cartty-corner 1}
t_1 < t_2 \quad \text{implies} \quad p_i(t_1) < p_i(t_2) , \forall i \in I
\end{equation}
The domain of a catty-corner curve can be any interval $J \subset \BR$, either open, closed or half-open. We also allow the domain of a catty-corner curve to extend up to $\pm \infty$, but in this case we will require the extra property that
\begin{equation}
\label{eqn:cartty-corner 2}
\lim_{t \rightarrow \pm \infty} p_i(t) = \pm \infty, \forall i \in I
\end{equation}

\medskip

\begin{proposition}
\label{prop:factor}

For any catty-corner curve $\bp : \BR \rightarrow \rr$, we have isomorphisms
\begin{equation}
	\label{eqn:factorization 1}	
	\bigotimes^{\rightarrow}_{t \in \BR} \CB^\pm_{\bp(t)} \xrightarrow{\sim} \CS^\pm \xleftarrow{\sim} \bigotimes^{\leftarrow}_{t \in \BR} \CB^\pm_{\bp(t)}  
\end{equation}
where $\rightarrow$ and $\leftarrow$ means that we take the tensor product in increasing and decreasing order of $t$, respectively. The isomorphisms in \eqref{eqn:factorization 1} are given by multiplication \footnote{As a vector space, $\otimes^{\rightarrow \text{ or } \leftarrow}_{t \in \BR} \CB^\pm_{\bp(t)}$ has a basis consisting of $\otimes^{\rightarrow \text{ or } \leftarrow}_{t \in \BR} b^\pm_{\bp(t), s_t}$, where $b^\pm_{\bp(t),s_t}$ run over homogeneous bases of $\CB^\pm_{\bp(t)}$ such that $b^\pm_{\bp(t),s_t} = 1$ for all but finitely many $t \in \BR$.}.

\end{proposition}

\medskip

\begin{proof} (following \cite[Proposition 3.14]{N R-matrix}). We will prove the $\rightarrow$ statement for $\pm = +$, and leave the other cases as exercises to the reader. We start by proving by induction on $\bn \in \nn$ that the multiplication map induces a surjection
\begin{equation}
\label{eqn:iso proof}
\mathop{\sum_{\{\bn_t \in \nn\}_{t \in \BR}, \text{ all but finitely many}}}_{\text{of which are }\b0, \text{ and }\sum_{t\in \BR} \bn_t = \bn} \ \bigotimes^\rightarrow_{t \in \BR} \CB_{\bp(t)|\bn_t} \twoheadrightarrow \CS_{\bn}
\end{equation}
To this end, note that the two assumptions that make a catty-corner curve (the fact that it is increasing and that it goes to infinity in the positive/negative orthant of $\rr$) imply that for any $(\bn,d) \in (\nn \backslash \b0) \times \BZ$ there exists a unique $t \in \BR$ such that
$$
\bp(t) \cdot \bn = d
$$
We will call this value of $t$ the \textbf{naive slope} (with respect to the curve $\bp(t)$) of any shuffle element of degree $(\bn,d)$. It is easy to see that having naive slope $<, \leq, =, \geq, >$ a certain real number is a multiplicative property. With this in mind, we note that the defining property of $\CS_{\geq \bp(t)}^+$, $\CS_{\leq \bp(t)}^+$ and $\CB^+_{\bp(t)}$ is that they consist of shuffle elements $E$ such that
\begin{align} 
&\Delta(E) \in \Big(\text{naive slope}\geq t \Big) \ \woo \ \Big(\text{anything} \Big) \label{eqn:naive 1} \\
&\Delta(E) \in \Big(\text{anything} \Big) \ \woo \ \Big(\text{naive slope}\leq t \Big) \label{eqn:naive 2} 
\end{align}
and
\begin{equation}
\label{eqn:naive 3}
\Delta(E) \in \Big(\text{naive slope}\geq t \Big) \otimes \Big(\text{naive slope}\leq t \Big)	
\end{equation}
respectively (we say that such $E$ have slope $\geq t$, $\leq t$ and $t$, respectively). In all relations above, we write $\CB_{\bp(\infty)}^+ = \BC[\ph_{i,d}^+]^{i \in I}_{d \geq 0}$, which is in line with the fact that the first tensor factor of $\Delta(E)$ contains elements from the middle subalgebra of \eqref{eqn:double shuffle}.

\medskip

\noindent Let us now consider an element $E \in \CS_{\bn}$ of naive slope $t$, and let us consider the largest $t'$ such that $\Delta(E)$ contains a non-zero summand $E_1 \otimes E_2$ with $E_2$ of naive slope $t'$. Among such summands, we assume that $E_2$ has maximal horizontal degree. It is clear that we must have $t' \geq t$; we may actually assume that $t' > t$, because if $t' = t$ then $E$ would already lie in the subalgebra $\CB_{\bp(t)}^+$. 

\medskip

\begin{claim}
\label{claim:1}
$E_1$ must have slope $<t'$ and $E_2$ must have slope $t'$.
\end{claim}

\medskip

\noindent Indeed, if $E_1$ failed to have slope $< t'$, then there would exist a summand in $\Delta(E_1)$ with second tensor factor of naive slope $\geq t'$, and the coassociativity of $\Delta$ would contradict the maximality of $t'$ and $E_2$. Similarly, if $E_2$ would fail to have slope $\leq t'$, then there would exist a summand in $\Delta(E_2)$ with second tensor factor of naive slope $>t'$, and the coassociativity of $\Delta$ would contadict the maximality of $t'$. Since having slope $\leq t'$ and naive slope $t'$ implies having slope $t'$, Claim \ref{claim:1} follows.

\medskip

\begin{claim}
\label{claim:2}

If $E'$ has slope $t'$ and $E''$ has slope $<t'$, then
\begin{equation}
\label{eqn:claim}
\Delta(E'' * E') \in E'' \kappa_{\emph{hdeg} E'}  \otimes E' + 
\end{equation}
$$
+ \Big(\text{anything}\Big) \otimes \Big(``\text{naive slope} < t'" \text{ or } ``\text{naive slope }t' \text{ and }\text{hdeg} < \hdeg E' " \Big)
$$

\end{claim}

\medskip 

\noindent This claim is an easy consequence of the fact that coproduct is multiplicative, and the assumptions made imply that
\begin{align*} 
&\Delta(E') \in \Big(\text{naive slope}\geq t' \Big) \otimes \Big(\text{naive slope}\leq t' \Big) \\
&\Delta(E'') \in \Big(\text{anything} \Big) \otimes \Big(\text{naive slope} < t' \Big) 
\end{align*} 
Thus, only the summand $E'' \otimes 1$ from $\Delta(E'')$ can produce terms which have second tensor factor of naive slope $t'$. This establishes Claim \ref{claim:2}.

\medskip

\noindent We will now prove the induction step of the isomorphism \eqref{eqn:iso proof}. Finitely many applications of Claims \ref{claim:1} and \ref{claim:2} allow us to construct
$$
\tilde{E} = E - X'' * X' - \dots - Y'' * Y' 
$$
(where $X'', \dots,Y''$ have slope $<t'$ and $X', \dots,Y'$ have slope $t'$) with $\Delta(\tilde{E})$ only having second tensor factors of naive slope $<t'$. We may then repeat this algorithm for naive slopes $t'' \in (t,t')$ (while there are infinitely many such $t''$, only finitely many of them can take part in this algorithm, since naive slope is constrained by the equation $\bp(t'') \cdot \bm \in \BZ$ for $\b0 < \bm \leq \bn$) to construct 
$$
\bar{E} = E - X'' * X' - \dots - Y'' * Y' - \dots - Z'' * Z' - \dots - T'' * T'
$$
such that $\Delta(\bar{E})$ only has second tensor factors of naive slope $\leq t$ . Therefore, after finitely many subtractions of products of the form $E'' * E'$ with $\text{slope}(E'') < \text{slope}(E')$ (all of these products are in the image of the map \eqref{eqn:iso proof} by the induction hypothesis) then we can transform $E$ into an element $\overline{E}$ of naive slope $t$, all of whose second tensor factors have naive slope $\leq t$. As we have already seen, such an element $\overline{E}$ must lie in $\CB_{\bp(t)|\bn}$, thus establishing the induction step of the surjectivity of \eqref{eqn:iso proof}.

\medskip

\noindent To show that the map \eqref{eqn:iso proof} is injective, we also invoke Claim \ref{claim:1} as follows: suppose for the purpose of contradiction there exists a non-trivial relation of the form
\begin{equation}
\label{eqn:non-trivial relation}
\sum \text{coefficient} \cdot E_{t_1}^{(k_1)} * \dots * E_{t_u}^{(k_u)} = 0
\end{equation} 
where $E_{t_s}^{(k_{t_s})}$ go over fixed homogeneous bases of $\CB_{\bp(t_s)}^+$, of horizontal degree $> \b0$, and all terms in \eqref{eqn:non-trivial relation} have the property that $t_1<\dots<t_u$. We assume that the total horizontal degree of \eqref{eqn:non-trivial relation} is minimal so that a non-trivial relation exists. If we consider the summands in the relation above with maximal $t_u$, and among those summands we consider those with maximal hdeg, then Claim \ref{claim:2} implies that
$$
\sum_{\text{maximal summands}} \text{coefficient} \cdot \Big( E_{t_1}^{(k_1)} * \dots * E_{t_{u-1}}^{(k_{u-1})} \Big) \otimes E_{t_u}^{(k_u)} = 0
$$
By the minimality of the horizontal degree in \eqref{eqn:non-trivial relation}, there are no non-trivial relations between the shuffle elements in the parenthesis above, which contradicts \eqref{eqn:non-trivial relation}. \end{proof}

\medskip

\subsection{Factorizations II}

It is easy to see that the argument in the proof of Proposition \ref{prop:factor} actually proves the following analogue of  \eqref{eqn:factorization 1} 
\begin{equation}
\label{eqn:factorization 2}
\bigotimes^{\rightarrow }_{t \in (t_1,t_2)} \CB^\pm_{\bp(t)} \xrightarrow{\sim} \CS^\pm_{> \bp(t_1)} \cap  \CS^\pm_{< \bp(t_2)} \xleftarrow{\sim} \bigotimes^{\leftarrow }_{t \in (t_1,t_2)} \CB^\pm_{\bp(t)}
\end{equation}
for any $t_1 < t_2$ (including $\pm \infty$), as well as the natural analogues of \eqref{eqn:factorization 2} when some endpoints of the intervals may be closed instead of open (see \cite[Proposition 3.14]{N R-matrix} for a proof of this more general statement in a related context). Because we can always find a catty-corner curve passing through any pair of points $\bp_1 < \bp_2$ in $\rr$, \eqref{eqn:factorization 2} implies that the multiplication map induces isomorphisms
\begin{equation}
\label{eqn:factorization 3}
\left(  \CS^\pm_{> \bp} \cap  \CS^\pm_{\leq \bp'} \right) \otimes \left(  \CS^\pm_{> \bp'} \cap  \CS^\pm_{< \bp''} \right) \xrightarrow{\sim} \CS^\pm_{> \bp} \cap  \CS^\pm_{< \bp''} \xleftarrow{\sim} \left(  \CS^\pm_{> \bp'} \cap  \CS^\pm_{< \bp''} \right) \otimes \left(  \CS^\pm_{> \bp} \cap  \CS^\pm_{\leq \bp'} \right)
\end{equation}
for any $\bp \leq \bp' < \bp''$ in $\rr$ (we allow $\bp = -\binfty$ or $\bp'' = \binfty$, in which case we make the convention that $\CS_{>-\binfty}^\pm = \CS_{<\binfty}^\pm = \CS^\pm$ in \eqref{eqn:factorization 3}).

\medskip

\begin{remark}
\label{rem:stable}

For simply laced $\fg$, we expect that $\CB_{\bp}$ correspond to the slope subalgebras defined in \cite{OS} using stable envelopes of Nakajima quiver varieties, and that the factorization \eqref{eqn:factorization 1} corresponds to the one introduced in loc. cit. using the FRT formalism (see \cite{Zhu 1, Zhu 2} for the case of finite and affine type A).

\end{remark}

\medskip

\noindent The left-most factorization in \eqref{eqn:factorization 1} respects the pairing \eqref{eqn:pairing shuffle}, in the sense that
\begin{equation}
\label{eqn:pair slopes basic}
\left \langle \prod_{t \in \BR}^{\rightarrow} E_t, \prod_{t \in \BR}^{\rightarrow} F_t \right \rangle = \prod_{t \in \BR} \langle E_t, F_t \rangle 
\end{equation}
for all $\{E_t \in \CB^+_{\bp(t)}, F_t \in \CB^-_{\bp(t)}\}_{t \in \BR}$ (almost all of which are 1). The proof of \eqref{eqn:pair slopes basic} is exactly the same as that of \cite[Proposition 3.12]{N R-matrix}, which dealt with $\bp(t) = t \bone$. Formula \eqref{eqn:pair slopes basic} and the non-degeneracy of the pairing \eqref{eqn:pairing shuffle} implies that the restriction
\begin{equation}
\label{eqn:really restricted 2}
\CB_{\bp}^{+} \otimes \CB_{\bp}^{-} \rightarrow \BC
\end{equation}
is also non-degenerate for any $\bp \in \rr$.

\medskip

\subsection{The half subalgebras}
\label{sub:slopes}

The main use of slopes for our purposes is given by 
\begin{align}
	&\CA^\geq =  \CS^-_{<\b0}  \otimes \BC[\ph_{i,0}^+,\ph_{i,1}^+,\ph_{i,2}^+,\dots]_{i \in I} \otimes \CS^+_{\geq \b0} \label{eqn:slope zero plus} \\
	&\CA^\leq = \CS^+_{<\b0}  \otimes \BC[\ph_{i,0}^-,\ph_{i,1}^-,\ph_{i,2}^-,\dots]_{i \in I} \otimes \CS^-_{\geq \b0} \label{eqn:slope zero minus}
\end{align}
which (\cite[Propositions 3.18, 3.21]{N Cat}) are subalgebras of $\CS$ and provide factorizations
\begin{equation}
	\label{eqn:a triangular}
	\CA^{\geq} \otimes \CA^{\leq} \xrightarrow{\sim} \CS \xleftarrow{\sim} \CA^{\leq} \otimes \CA^{\geq}
\end{equation}
The statements above would hold with $\b0$ replaced by any $\bp \in \rr$. However, the importance of using $\b0$ stems from the following result (\cite[Proposition 3.23]{N Cat}).

\medskip

\begin{proposition}
	\label{prop:slope affine}
	
	If $\fg$ is of finite type, then there is an isomorphism
	\begin{equation}
		\label{eqn:xi}
		\Xi : \UUaff \xrightarrow{\sim} \CS
	\end{equation}
	that sends the subalgebra $\UUaffg$ onto $\CA^\geq$ and  $\UUaffl$ onto $\CA^\leq$ (see \cite[Section 2]{N Cat} for a review of quantum affine algebras in our notation).
	
\end{proposition}

\bigskip

\section{Dimensions}
\label{sec:dimensions}

\medskip

\subsection{Dimension of slope subalgebras I}

For any Kac-Moody Lie algebra $\fg$, let us define the integers  $\{a_{\fg,\bn}\}_{\bn \in \nn}$ recursively by the following formula 
\begin{equation}
	\label{eqn:conj hor}
	\sum_{\bn \in \nn} \dim_{\BC}(\CB_{\b0|\bn}) q^{\bn} = \prod_{\bn \in \nn \backslash \b0} \left( \frac 1{1-q^{\bn} } \right)^{a_{\fg,\bn}}
\end{equation}
where $\{q^{\bn}\}_{\bn \in \zz}$ are formal symbols such that
\begin{equation}
\label{eqn:formal symbols}
q^{\bn+\bm} = q^{\bn} q^{\bm}
\end{equation}
for all $\bn,\bm \in \zz$. 

\medskip

\begin{remark} Although beyond the scope of the present paper, $\CB_{\b0}$ is a $q$-Borcherds algebra, and therefore the PBW theorem implies that $a_{\fg,\bn}$ are non-negative integers. \end{remark}

\medskip

\noindent For the remainder of the present paper, we make the following assumption
\begin{equation}
	\label{eqn:conj}
	\sum_{\bn \in \nn} \sum_{d=0}^{\infty} \dim_{\BC} \left(\CS_{\geq \b0|\bn,d} \right) q^{\bn} v^d \stackrel{?}= \prod_{\bn \in \nn \backslash \b0} \prod_{d=0}^{\infty} \left( \frac 1{1- q^{\bn} v^d} \right)^{a_{\fg,\bn}}
\end{equation}
This conjectural formula holds when $\fg$ is of finite type, as shown in \cite[formula (5.43)]{N Cat} (in this case, $a_{\fg,\bn}$ is 1 if $\bn$ is a positive root of $\fg$ and 0 otherwise). 

\medskip

\begin{proposition}
\label{prop:slope dim}

Subject to \eqref{eqn:conj}, we have for any $\bp \in \rr$
\begin{equation}
\label{eqn:prop hor}
\mathop{\sum_{\bn \in \nn}}_{\text{s.t. }\bp\cdot \bn \in \BZ} \dim_{\BC}(\CB_{\bp| \bn}) q^{ \bn} v^{\bp \cdot \bn} = \mathop{\prod_{\bn \in \nn \backslash \b0}}_{\text{s.t. } \bp \cdot \bn \in \BZ} \left( \frac 1{1-q^{\bn} v^{\bp \cdot \bn}} \right)^{a_{\fg,\bn}}
\end{equation}

\end{proposition}

\medskip

\begin{proof} By applying a shift automorphism $\sigma_{\br}$ with $\br = (r_i \gg 0)_{i \in I}$ and invoking \eqref{eqn:sigma on b}, we may assume that all the entries of $\bp$ are positive (when performing the shift by $\br$, we also change the variables according to $q^{\bn} \mapsto q^{\bn} v^{\br\cdot \bn}$ in \eqref{eqn:prop hor}). In this case, let us consider the factorization \eqref{eqn:factorization 2} for the straight line $\bp(t) = t \bp$
$$
\bigotimes_{t \in [0,\infty)}^{\rightarrow} \CB_{t\bp}^+ \stackrel{\sim}\longrightarrow \CS^+_{\geq \b0}
$$
Then formula \eqref{eqn:conj} implies that
\begin{equation}
\label{eqn:graded dimension full}
\prod_{t \in [0,\infty)} \left( \mathop{\sum_{\bn \in \nn}}_{\text{s.t. }t \bp \cdot \bn \in \BN} \dim_{\BC} \left(\CB_{t \bp|\bn} \right) q^{\bn} v^{t \bp \cdot \bn} \right) = \mathop{\prod_{\bn \in \nn \backslash \b0}}_{d \in \BN} \left( \frac 1{1- q^{\bn} v^d} \right)^{a_{\fg,\bn}}
\end{equation}
We will show that the formula above implies that for all $t \geq 0$, we have
\begin{equation}
\label{eqn:graded dimension}
\mathop{\sum_{\bn \in \nn}}_{\text{s.t. }t \bp \cdot \bn \in \BN} \dim_{\BC} \left(\CB_{t \bp|\bn} \right) q^{\bn} v^{t \bp \cdot \bn} = \mathop{\prod_{\bn \in \nn \backslash \b0}}_{\text{s.t. }d = t\bp \cdot \bn \in \BN} \left( \frac 1{1- q^{\bn} v^d} \right)^{a_{\fg,\bn}}
\end{equation}
(the $t=1$ case of which is \eqref{eqn:prop hor}, and so concludes the proof of Proposition \ref{prop:slope dim}). Indeed, assume the contrary and choose a minimal $\bm \in \nn$ (with respect to the partial order in Subsection \ref{sub:basic}) such the LHS and RHS of \eqref{eqn:graded dimension} differ in horizontal degree $\bm$, and then choose a minimal $t' \in [0,\infty)$ for which such a difference occurs (this is well-defined because the set of $t' \geq 0$ for which $\CB_{t' \bp|\bn} \neq 0$ for $\bn \leq \bm$ is discrete). Then \eqref{eqn:graded dimension full} together with \eqref{eqn:graded dimension} for $t < t'$ imply that
\begin{equation}
\label{eqn:up to}
\prod_{t \in [t',\infty)} \left( \mathop{\sum_{\bn \in \nn \text{ s.t.}}}_{t \bp \cdot \bn \in \BN} \dim_{\BC} \left(\CB_{t \bp|\bn} \right) q^{\bn} v^{t \bp \cdot \bn} \right) =_{\bm} \mathop{\prod_{\bn \in \nn \backslash \b0}}_{d \in \BZ_{\geq t' \bp \cdot \bn}} \left( \frac 1{1- q^{\bn} v^d} \right)^{a_{\fg,\bn}}
\end{equation}
where $=_{\bm}$ means that the coefficients of all $\{q^{\bn}\}_{\bn \leq \bm}$ in the LHS and RHS match. However, only the $t=t'$ term of the LHS and the $d=t'\bp \cdot \bn$ term of the RHS contribute to the $q^{\bm} v^{t' \bp \cdot \bm}$ term of equation \eqref{eqn:up to}, which contradicts our choice of $\bm$ and $t'$. \end{proof}

\medskip

\subsection{Dimension of slope subalgebras II}

We still operate under assumption \eqref{eqn:conj}. Having established \eqref{eqn:prop hor} for all $\bp \in \rr$, the factorization \eqref{eqn:factorization 2} implies that
\begin{equation}
\label{eqn:formula}
\sum_{\bn \in \nn} \sum_{d \in \BZ} \dim_{\BC} \left(\CS_{> \bp_1|\bn,d} \cap \CS_{< \bp_2|\bn,d}\right) q^{\bn} v^d = \mathop{\prod_{\bn \in \nn \backslash \b0}}_{d \in \BZ_{>\bp_1 \cdot \bn} \cap \BZ_{< \bp_2 \cdot \bn}} \left( \frac 1{1- q^{\bn} v^d} \right)^{a_{\fg,\bn}}
\end{equation}
for all $\bp_1 < \bp_2$ in $\rr$ (note that both sides of the equation above are independent on the choice of catty-corner curve which goes into the factorization \eqref{eqn:factorization 2}). The formula above also holds when $\bp_1 = -\binfty$ exclusive or $\bp_2 = \binfty$ (with the RHS expanded in either $v^{-1}$ or $v$) and when we replace $>$ and/or $<$ by $\geq$ and/or $\leq$ in both sides of the equation. Finally, note that the identity assignment on shuffle elements yields an algebra anti-automorphism 
$$
\CB_\bp^+ \stackrel{\sim}\rightarrow \CB_{-\bp}^-
$$
which implies the natural analogues of \eqref{eqn:conj hor}, \eqref{eqn:conj}, \eqref{eqn:prop hor}, \eqref{eqn:formula} for $\CS^-$ instead of $\CS^+$.

\medskip

\begin{definition}
\label{def:generic}
	
We call $\bp \in \rr$ \textbf{generic} if
$$
\Big\{\bn \in \zz \Big| \bp \cdot \bn \in \BZ \Big\} = \begin{cases} 0 &\text{or } \\ \BZ \bm &\text{for some }\bm \in \zz \backslash \b0 \end{cases}
$$
A catty-corner curve $\bp(t)$ will be called \textbf{generic} if $\bp(t)$ is generic for all $t$ (the terminology is due to the fact that such a curve does not pass through any intersection of two hyperplanes of the form $\bp \cdot \bn = d$ for various $\bn \in \nn \backslash \b0$ and $d \in \BZ$).

\medskip 

\end{definition}

\noindent Proposition \ref{prop:slope dim} shows that if $\bp$ is generic, then
\begin{equation}
\label{eqn:simple prop}
\sum_{\bn \in \nn} \dim_{\BC}(\CB_{\bp| \bn}) q^{ \bn} v^{\bp \cdot \bn} = \prod_{k=1}^{\infty} \left( \frac 1{1-q^{k \bm} v^{k \bp \cdot \bm}} \right)^{a_{\fg, k\bm}}
\end{equation}
for some fixed $\bm \in \nn$. If moreover $\fg$ is of finite type, the fact that $a_{\fg,\bn}$ is non-zero only if $\bn$ is a positive root (in which case $a_{\fg,\bn}=1$) shows that $\CB_{\bp}$ has the same graded dimension as $U_q(\fsl_2)$ (specifically, dimension 1 in horizontal degrees $\bn \in \BN \bm$ and dimension 0 in all other horizontal degrees). We will show in Subsection \ref{sub:identifying} that these algebras are in fact isomorphic.

\medskip

\subsection{Pairing}

Formula \eqref{eqn:pair slopes basic} deals with the orthogonality of the pairing with respect to factorizations of $\CS^+$ and $\CS^-$ defined with respect to one and the same catty-corner curve. However, the following result pertains to a more general setting.

\medskip

\begin{proposition}
\label{prop:key}

Subject to \eqref{eqn:conj}, for any generic $\bp_1 < \bp_2$ in $\rr$, consider the following restriction \footnote{The notation $\langle\cdot,\cdot\rangle_{\bp_1,\bp_2}$ is only for bookkeeping reasons; the pairing in \eqref{eqn:p1 p2} is none other than the appropriate restriction of the usual pairing $\CS^+ \otimes \CS^- \rightarrow \BC$.} of the pairing \eqref{eqn:pairing shuffle} twisted by the antipode 
\begin{equation}
\label{eqn:p1 p2}
\CS_{\geq \b0}^+ \otimes \Big(\CS_{>\bp_1}^- \cap \CS_{<\bp_2}^- \Big)\xrightarrow{\langle\cdot,S(\cdot)\rangle_{\bp_1,\bp_2}} \BC
\end{equation}
We have
\begin{equation}
\label{eqn:key}
\sum_{\bn \in \nn} \sum_{d=0}^{\infty} \dim_{\BC} \left( \frac {\CS_{\geq \b0|\bn,d}}{\emph{Ker } \langle\cdot,S(\cdot)\rangle_{\bp_1,\bp_2}} \right) q^{\bn}v^d = \mathop{\prod_{\bn \in \nn \backslash \b0}}_{d \in \BZ_{> \bp_1 \cdot \bn} \cap \BZ_{<\bp_2 \cdot \bn} \cap \BN}  \left( \frac 1{1- q^{\bn} v^d} \right)^{a_{\fg,\bn}}
\end{equation}

\end{proposition}

\begin{proof} Because $\bp_1< \bp_2$ are generic, we may choose a generic catty-corner curve
$$
\bp: (-\infty,\infty) \rightarrow \rr
$$
such that $\bp(1) = \bp_1$ and $\bp(2) = \bp_2$. Therefore, we will write $\text{LHS}_{t_1,t_2}$ and $\text{RHS}_{t_1,t_2}$ for the left and right-hand sides of \eqref{eqn:key} with $\bp_1$ and $\bp_2$ replaced by $\bp(t_1)$ and $\bp(t_2)$, respectively. We will show that $\text{LHS}_{t_1,t_2} = \text{RHS}_{t_1,t_2}$ in all horizontal degrees $\bn \leq \bm$ for any henceforth fixed $\bm \in \nn$, which would conclude the proof of Proposition \ref{prop:key}.

\medskip 

\noindent Note that the algebras $\CS_{>\bp(t)}^-$ and $\CS^-_{<\bp(t)}$ only change when the catty-corner curve $\bp(t)$ encounters a hyperplane of the form $\bp \cdot \bn \in \BZ$. The genericity of the catty-corner curve implies that any such encounter is with a unique such hyperplane. Since there are discretely many such hyperplanes with $\bn$ less than or equal to our fixed $\bm$, then we conclude that $\text{LHS}_{t_1,t_2}$ only changes at the discretely many values $\bp(t_1) \cdot \bn \in \BZ$ and $\bp(t_2) \cdot \bn \in \BZ$. Moreover, $\text{RHS}_{t_1,t_2}$ also only changes at the same discretely many values of $t_1$ and $t_2$. Therefore, let us consider how $\text{LHS}_{t_1,t_2}$ and $\text{RHS}_{t_1,t_2}$ change when we encounter such a value of $t$. Explicitly, assume that
$$
\bp(t_1) \cdot \bn_1 = d_1 \in \BZ \quad \text{and} \quad \bp(t_2) \cdot \bn_2 = d_2 \in \BZ
$$
(and we take $\bn_1, \bn_2 \neq \b0$ to be minimal with respect to the existence of $d_1,d_2 \in \BZ$ as above). It is then easy to see that for any small enough $\varepsilon > 0$
\begin{equation}
\label{eqn:change rhs 1}
\text{RHS}_{t_1-\varepsilon,t_2} = \text{RHS}_{t_1+\varepsilon,t_2} \prod_{k = 1}^{\infty}\left( \frac 1{1- q^{k\bn_1} v^{kd_1}} \right)^{\delta_{d_1 \geq 0} a_{\fg,k\bn_1}}
\end{equation}
\begin{equation}
\label{eqn:change rhs 2}
\text{RHS}_{t_1, t_2+\varepsilon} = \text{RHS}_{t_1,t_2-\varepsilon} \prod_{k = 1}^{\infty}\left( \frac 1{1- q^{k\bn_2} v^{kd_2}} \right)^{\delta_{d_2 \geq 0} a_{\fg,k\bn_2}}
\end{equation}
In the subsequent paragraph, we will prove the following formulas
\begin{equation}
\label{eqn:change lhs 1}
\text{LHS}_{t_1-\varepsilon,t_2} \leq \text{LHS}_{t_1+\varepsilon,t_2} \left( \sum_{k=0}^{\infty} \dim_{\BC} \left(\CB_{\bp(t_1)|k\bn_1} \right) q^{k\bn_1} v^{k d_1} \right)^{\delta_{d_1 \geq 0}}
\end{equation}
\begin{equation}
\label{eqn:change lhs 2}
\text{LHS}_{t_1,t_2+\varepsilon} \leq \text{LHS}_{t_1,t_2-\varepsilon} \left( \sum_{k=0}^{\infty} \dim_{\BC} \left(\CB_{\bp(t_2)|k\bn_2} \right) q^{k\bn_2} v^{k d_2} \right)^{\delta_{d_2 \geq 0}}
\end{equation}
but let us first show how they allow us to conclude the proof of Proposition \ref{prop:key}. By Proposition \ref{prop:slope dim}, the products in \eqref{eqn:change rhs 1}-\eqref{eqn:change rhs 2} are equal to the respective expressions in parentheses in \eqref{eqn:change lhs 1}-\eqref{eqn:change lhs 2}. The fact that $\text{LHS}_{t,t}$ and $\text{RHS}_{t,t}$ are both equal to 1 allows us to  conclude (by induction on $t_2-t_1$, which is well-defined, since it is enough to consider discretely many values of $t_1$ and $t_2$ in the present argument) that 
\begin{equation}
\label{eqn:inequality}
\text{LHS}_{t_1,t_2} \leq \text{RHS}_{t_1,t_2}
\end{equation}
for all $t_1 < t_2$. However, because the pairing \eqref{eqn:pairing shuffle} is non-degenerate, we have 
$$
\frac {\CS_{\geq \b0}^+}{\text{Ker } \langle\cdot,S(\cdot)\rangle_{\bp(-\infty),\bp(\infty)}} = \CS_{\geq \b0}^+
$$
and therefore $\text{LHS}_{-\infty,\infty} = \text{RHS}_{-\infty,\infty}$ due to \eqref{eqn:conj}. We therefore conclude that the inequality \eqref{eqn:inequality} is an equality for all $t_1 < t_2$. In particular, when $t_1 = 1$ and $t_2 = 2$, this equality implies \eqref{eqn:key}. 

\medskip

\noindent We will only prove \eqref{eqn:change lhs 1}, as \eqref{eqn:change lhs 2} is completely analogous. Consider first the case $d_1 \geq 0$. We note that the LHS of \eqref{eqn:key} is the graded dimension of the vector space of linear maps
$$
\left\{\lambda : \CS_{\geq \b0}^+ \rightarrow \BC \Big| \lambda(-) = \langle -, S(x) \rangle \text{ for some } x \in \CS^-_{>\bp_1} \cap \CS^-_{<\bp_2} \right\}
$$
Therefore, the LHS of \eqref{eqn:change lhs 1} counts the graded dimension of 
$$
\left\{\lambda : \CS_{\geq \b0}^+ \rightarrow \BC \Big| \lambda(-) = \langle -, S(x) \rangle \text{ for some } x \in \CS^-_{>\bp(t_1-\varepsilon)} \cap \CS^-_{<\bp(t_2)} \right\}
$$
while the RHS of \eqref{eqn:change lhs 1} counts the graded dimension of 
\begin{multline*} 
\left\{\mu : \CS_{\geq \b0}^+ \rightarrow \BC \Big| \mu(-) = \langle -, S(y) \rangle \text{ for some } y \in \CS^-_{>\bp(t_1+\varepsilon)} \cap \CS^-_{<\bp(t_2)} \right\} \otimes \\ \underbrace{\otimes \left\{\nu : \CS^+ \rightarrow \BC \Big| \nu(-) = \langle -, S(z) \rangle \text{ for some } z \in \CB^-_{\bp(t_1)} \right\}}_{\text{has the same dimension as }\CB_{\bp(t_1)}^+, \text{ due to the non-degeneracy of \eqref{eqn:really restricted 2}}}
\end{multline*} 
By \eqref{eqn:factorization 3}, multiplication induces $(\CS^-_{>\bp(t_1-\varepsilon)} \cap \CS^-_{<\bp(t_2)}) = \CB_{\bp(t_1)}^- \otimes (\CS^-_{>\bp(t_1+\varepsilon)} \cap \CS^-_{<\bp(t_2)})$, so the inequality \eqref{eqn:change lhs 1} follows from the dimension counts above and the fact that
$$
\langle E, S(zy) \rangle \stackrel{\eqref{eqn:bialgebra 1}}= \langle \ph E_1,S(y)\rangle \langle E_2,S(z)\rangle \stackrel{\eqref{eqn:bialgebra 2},\eqref{eqn:antipode pairing}}= \langle E_1,S(y)\rangle \langle E_2,S(z)\rangle
$$
for all $E \in \CS_{\geq \b0}^+$ with $\Delta(E) = \ph E_1 \otimes E_2$ (where $\ph = \ph^+_{i_1,d_1} \ph^+_{i_2,d_2} \dots$ and $E_1 \in \CS_{\geq \b0}^+$). 

\medskip 

\noindent Now let us prove \eqref{eqn:change lhs 1} when $d_1 < 0$. We will actually show that $\text{LHS}_{t_1-\varepsilon,t_2} = \text{LHS}_{t_1+\varepsilon,t_2}$, or equivalently, we will prove the equality of the graded dimensions of 
\begin{equation}
\label{eqn:vel 1}
\left\{\lambda : \CS_{\geq \b0}^+ \rightarrow \BC \Big| \lambda(-) = \langle -, S(x) \rangle \text{ for some } x \in \CS^-_{>\bp(t_1-\varepsilon)} \cap \CS^-_{<\bp(t_2)} \right\}
\end{equation}
and
\begin{equation}
\label{eqn:vel 2}
\left\{\mu : \CS_{\geq \b0}^+ \rightarrow \BC \Big| \mu(-) = \langle -, S(y) \rangle \text{ for some } y \in \CS^-_{>\bp(t_1+\varepsilon)} \cap \CS^-_{<\bp(t_2)} \right\}
\end{equation}
By \eqref{eqn:factorization 3}, multiplication induces $(\CS^-_{>\bp(t_1-\varepsilon)} \cap \CS^-_{<\bp(t_2)}) = (\CS^-_{>\bp(t_1+\varepsilon)} \cap \CS^-_{<\bp(t_2)}) \otimes \CB_{\bp(t_1)}^-$, and so any $x$ that appears in \eqref{eqn:vel 1} can be written uniquely as a linear combination of $yz$, for various $y$ as in \eqref{eqn:vel 2} and $z \in \CB^-_{\bp(t_1)}$. The equality of the dimensions of the vector spaces \eqref{eqn:vel 1} and \eqref{eqn:vel 2} follows immediately from the fact that
$$
\langle E, S(yz) \rangle \stackrel{\eqref{eqn:bialgebra 1}}= \langle \ph E_1,S(z)\rangle \langle E_2,S(y)\rangle  = \varepsilon(z) \langle E, S(y) \rangle 
$$
The second equality is due to the fact that $E_1 \in \CS^+_{\geq \b0}$ pairs trivially with $z \in \CB_{\bp(t_1)}^-$ unless $\hdeg E_1 = \hdeg z = \b0$ (this is an immediate consequence of $d_1 < 0$; recall that the pairing is non-zero only on elements of opposite degrees). \end{proof}

\medskip

\subsection{Identifying the subalgebras}
\label{sub:identifying}

In the present Subsection, we assume that $\fg$ is of finite type. For any generic $\bp \in \rr$, we showed in \eqref{eqn:simple prop} that $\CB_{\bp}$ is either a trivial algebra (if $\bp \cdot (\nn \backslash \b0) \notin \BZ$) or it has the same graded dimension as $U_q(\fsl_2)$ (if there exists $\bn \in \nn \backslash \b0$ such that $\bp \cdot \bn \in \BZ$, in which case such $\bn$ would range over all multiples of a positive root $\balpha$). In the latter case, we claim that we have a Hopf algebra isomorphism
\begin{equation}
\label{eqn:iso sl2}
U_{q_\balpha}(\fsl_2) \xrightarrow{\sim} \CB_{\bp}
\end{equation}
where $q_\balpha = q^{\frac {(\balpha,\balpha)}2}$, defined on the standard generators of $U_{q_{\balpha}}(\fsl_2)$ as follows
$$
e \mapsto E_{\balpha}, \qquad f \mapsto F_{\balpha}, \qquad \kappa \mapsto \kappa_{\balpha}
$$
with $E_{\balpha}, F_{\balpha}$ being basis vectors of the one-dimensional vector spaces $\CB_{\bp|\balpha}, \CB_{\bp|-\balpha}$, respectively (we normalize these vectors so that the Hopf pairing between them has the same value as the Hopf pairing between the standard elements of quantum $\fsl_2$). The isomorphism \eqref{eqn:iso sl2} follows from the fact that $\CB_{\bp}$ is the Drinfeld double of \footnote{Note that the algebra below slightly differs from that of \eqref{eqn:enlarged slope}, since we only added the Cartan element $\kappa_{\balpha}$ instead of all the Cartan elements $\{\kappa_i\}_{i \in I}$.}
$$
\CB_{\bp}^\geq = \CB_{\bp}^+ \otimes \BC[\kappa_\balpha] \quad \text{and} \quad \CB_{\bp}^\leq = \BC[\kappa^{-1}_\balpha] \otimes \CB_{\bp}^-
$$
(see Proposition \ref{prop:slope subalgebra}) and the following result. 

\medskip

\begin{claim}
\label{claim:3} 
	
Suppose we have algebras $B^\pm = \bigoplus_{k=0}^{\infty} B_{\pm k}$ with $\dim_{\BC}(B_{k})=1$ for all $k \in \BZ$ (we will call $k$ the grading of these algebras), and bialgebra structures on  
$$
B^{\geq} = \frac {B^+ \otimes \BC[\kappa]}{\kappa b = q^{2\deg b} b \kappa}, \qquad B^{\leq} = \frac {\BC[\kappa^{-1}] \otimes B^-}{\kappa^{-1} b = q^{-2\deg b} b \kappa^{-1}}
$$
such that 
\begin{align*}
&\Delta(b_+) = \kappa \otimes b_+ + b_+ \otimes 1 \\
&\Delta(b_-) = 1 \otimes b_- + b_- \otimes \kappa^{-1}
\end{align*} 
for any $b_\pm \in B_{\pm 1}$. Assume that there exists a bialgebra pairing
$$
B^\geq \otimes B^\leq \xrightarrow{\langle \cdot,\cdot \rangle}\BC
$$
that satisfies $\langle \kappa,\kappa^{-1}\rangle = q^2$ and whose restriction to $B^+ \otimes B^-$ is non-degenerate. Then there are bialgebra isomorphisms
\begin{equation}
\label{eqn:bialgebra isomorphisms}
U_q^{\geq}(\fsl_2) \xrightarrow{\sim} B^\geq \quad \text{and} \quad U_q^{\leq}(\fsl_2) \xrightarrow{\sim} B^\leq
\end{equation}
given by $e \mapsto b_+, f \mapsto b_-, \kappa \mapsto \kappa$.

\end{claim}

\noindent  Indeed, the fact that we have 
\begin{align*}
&\Delta_{\bp}(E_{\balpha}) = \kappa_\balpha \otimes E_{\balpha} + E_{\balpha} \otimes 1 \\
&\Delta_{\bp}(F_{\balpha}) = 1 \otimes F_{\balpha} + F_{\balpha} \otimes \kappa_{\balpha}^{-1}
\end{align*} 
holds because of \eqref{eqn:coproduct slope plus}-\eqref{eqn:coproduct slope minus} and the fact that $\bp \cdot \bn \notin \BZ$ for any $\bn < \balpha$. Thus, Claim \ref{claim:3} applies to the case at hand, establishing \eqref{eqn:iso sl2}.

\medskip 

\noindent To prove Claim \ref{claim:3}, it suffices to show that $b_\pm^k \neq 0$ for any $k > 0$, because then the coassociativity of the coproduct implies that $\Delta(b_\pm^k)$ matches $\Delta(e^k)$ under the assignment $e\mapsto b_\pm $. Then $e^k \mapsto b_+^k, f^k \mapsto b_-^k$ define bialgebra isomorphisms \eqref{eqn:bialgebra isomorphisms} and we would be done. However, the fact that $b_\pm^k \neq 0$ follows from the fact that
$$
\langle b_+^k, b_-^k \rangle
$$
is calculated using \eqref{eqn:bialgebra 1}-\eqref{eqn:bialgebra 2}, and expressed in terms of $\langle b_+,b_-\rangle^k$ and appropriate powers of $q^2 = \langle \kappa,\kappa^{-1}\rangle$. Therefore, the pairing above matches $\langle e^k, f^k \rangle$ in $U_q(\fsl_2)$, and is thus non-zero due to the non-degeneracy of the Hopf pairing of quantum $\fsl_2$.

\bigskip

\section{$q$-characters}

\medskip 

\subsection{Category $\CO$}
\label{sub:representations}

For any Kac-Moody Lie algebra $\fg$, we consider representations
\begin{equation}
\label{eqn:representation}
\CA^\geq \curvearrowright V
\end{equation}
of the half subalgebra \eqref{eqn:slope zero plus}. In \cite{N Cat}, we generalized the Borel category $\CO$ of \cite{HJ} from finite type $\fg$ to any Kac-Moody $\fg$, as follows. For any $\ell$-weight
\begin{equation}
\label{eqn:ell weight}
\bpsi = (\psi_i(z))_{i \in I} \in \left(\BC[[z^{-1}]]^* \right)^I
\end{equation}
we construct a simple module given by
\begin{equation}
\label{eqn:simple module}
L(\bpsi) \cong \CS_{<\b0}^- \Big/ J(\bpsi)
\end{equation}
Above, $J(\bpsi) = \oplus_{\bn \in \nn} J(\bpsi)_{\bn}$, where $J(
\bpsi)_{\bn}$ consists of those $F\in \CS_{<\b0|-\bn}$ such that
\begin{equation}
\label{eqn:j pairing}
\left \langle E(z_{i1},\dots,z_{in_i}) \prod_{i \in I} \prod_{a=1}^{n_i} \psi_i(z_{ia}), S\left(F(z_{i1},\dots,z_{in_i})\right) \right \rangle = 0
\end{equation}
for all $E \in \CS_{\geq \b0|\bn}$. The $\CA^\geq$-module structure on $L(\bpsi)$ is described in \cite[Proposition 4.5]{N Cat}. For the purposes of the present paper, the only things that we need to know about simple modules are their grading
\begin{equation}
\label{eqn:grading simple}
\CS_{<\b0}^- \Big / J(\bpsi) =  \bigoplus_{\bn \in \nn} \CS_{<\b0|-\bn} \Big / J(\bpsi)_{\bn}
\end{equation}
and the way that Cartan elements act on $L(\bpsi)$. In the language of \eqref{eqn:simple module}, we have
\begin{equation}
\label{eqn:cartan series}
\frac {\ph^+_j(y)}{\psi_j(y)} \quad \text{sends} \quad F(z_{i1},\dots,z_{in_i})_{i \in I} \quad \text{to}
\end{equation}
$$
F(z_{i1},\dots,z_{in_i})_{i \in I} \cdot q^{(-\bn,\bs^j)} \exp \left( \sum_{u=1}^{\infty} \sum_{i \in I} (z_{i1}^u + \dots + z_{in_i}^u ) \frac {q^{-ud_{ij}} - q^{ud_{ij}}}{uy^u} \right)
$$
which is an immediate consequence of \eqref{eqn:shuffle minus commute} and the fact that $\ph_j^+(y)$ acts on the cyclic generator $(1 \text{ mod } J(\bpsi)) \in L(\bpsi)$ by multiplication with $\psi_j(y)$. A representation \eqref{eqn:representation} is said to be in \textbf{category $\CO$} (\cite{HJ}) if its weight spaces
$$
V_{\bom} = \Big\{ v \in V \Big| \kappa_i \cdot v = q^{(\bom,\bs^i)}v, \ \forall i \in I \Big \}
$$
are all finite-dimensional, and non-zero only for $\bom$ lying in a finite union of translates of $- \nn$. From \eqref{eqn:grading simple}, it is clear that $L(\bpsi)$ lies in category $\CO$ if and only if \footnote{Technically speaking, this requires the symmetric bilinear form \eqref{eqn:symmetric pairing} to be non-degenerate, which is the case for finite type $\fg$, but not necessarily for Kac-Moody $\fg$. The usual workaround in the latter case is to enlarge the Cartan subalgebra so as to make the form \eqref{eqn:symmetric pairing} non-degenerate.} $\CS_{<\b0|-\bn}/J(\bpsi)_{\bn}$ is a finite-dimensional vector space for all $\bn \in \nn$. In turn, the finite-dimensionality of the aforementioned vector spaces was shown in \cite{HJ, N Cat} to be equivalent to every $\psi_i(z)$ being the power series expansion of a rational function, in which case we call the whole of $\bpsi$ \textbf{rational}. If we regard $\CS_{<\b0|-\bn}/J(\bpsi)_{\bn}$ as a module over the ring of color-symmetric polynomials in $\bn$ variables, its finite-dimensionality implies that it is supported at finitely many closed points
$$
\bx = (x_{ia})_{i \in I, 1\leq a \leq n_i} \in \prod_{i\in I} \BC^{n_i}/S_{n_i} = \BC^{\bn}
$$
We may therefore refine \eqref{eqn:grading simple} as follows
\begin{equation}
	\label{eqn:grading simple upgrade}
	L(\bpsi) \cong \bigoplus_{\bn \in \nn} \bigoplus_{\bx \in \BC^{\bn}} \left( \CS_{<\b0|-\bn} \Big / J(\bpsi)_{\bn} \right)_{\bx}
\end{equation}
where the vector space in parentheses in the RHS consists of elements on which multiplication by $z_{i1}^u+\dots+z_{in_i}^u - x_{i1}^u - \dots - x_{in_i}^u$ is nilpotent, for all $i \in I, u \geq 1$. 

\medskip

\subsection{The contribution of $\neq 0$}
\label{sub:non-zero}

As any $\bx \in \BC^{\bn}$ can be written uniquely as the concatenation of some $\by \in (\BC^*)^{\bm}$ with the origin $\b0_{\bn-\bm} \in \BC^{\bn-\bm}$ for some $\bm \leq \bn$, we will now show how to ``factor" the right-hand side of \eqref{eqn:grading simple upgrade} in these terms. We first start with the contribution of points $\by \in (\BC^*)^{\bm}$ with non-zero entries.

\medskip 

\begin{definition}
\label{def:non-zero}

For any $\ell$-weight $\bpsi$, define
$$
J^{\neq 0}(\bpsi) = \bigoplus_{\bn \in \nn} J^{\neq 0}(\bpsi)_{\bn}
$$
where $J^{\neq 0}(\bpsi)_{\bn}$ consists of those $F \in \CS_{<\b0|-\bn}$ such that 
\begin{equation}
	\label{eqn:j pairing new}
	\left \langle E(z_{i1},\dots,z_{in_i}) \prod_{i \in I} \prod_{a=1}^{n_i} \psi_i(z_{ia}), S\left(F(z_{i1},\dots,z_{in_i})\right) \right \rangle = 0
\end{equation}
for all $E \in \CS_{\geq  (N,\dots,N)|\bn}$ for some $N$ large enough (depending on $F$).

\end{definition}

\medskip 

\noindent We will abbreviate $N\bone = (N,\dots,N)$. While 
\begin{equation}
	\label{eqn:simple module new}
L^{\neq 0}(\bpsi) = \CS_{<\b0}^- \Big/ J^{\neq 0}(\bpsi)
\end{equation}
is not an $\CA^{\geq}$ module, it is a module for $\BC[\ph_{i,d}^+]_{i \in I, d\geq 0}$ via the action \eqref{eqn:cartan series}. The reason for this is that we have $\langle E\psi, S(FP)\rangle = \langle EP\psi,S(F)\rangle$ for any color-symmetric polynomial $P$ in $\bn$ variables, and $\CS_{\geq N \bone|\bn}$ is closed under multiplication with $P$.

\medskip

\begin{proposition}
\label{prop:residues}

We have $(\CS_{<\b0|-\bn} / J^{\neq 0}(\bpsi)_{\bn})_{\bx} \neq 0$ only if for all $(i,a)$ 

\medskip 

\begin{itemize}
	
	\item $x_{ia}$ is a non-zero pole of $\psi_i(z)$, or
	
	\medskip 
	
	\item $x_{ia} = x_{jb} q^{-d_{ij}}$ for some $(j,b) < (i,a)$
	
\end{itemize}

\medskip 

\noindent with respect to some total order on the set $\{(i,a) | i \in I, 1 \leq a \leq n_i\}$.	

\end{proposition}

\medskip

\begin{proof} For any $M \in \BN$, let 
\begin{equation}
\label{eqn:contain 1}
\oCS^+_{\geq M \bone} \subseteq \CS^+_{\geq M \bone}
\end{equation}
denote the subalgebra generated by $\{e_{i,d}\}_{i \in I, d\geq M}$. As follows from \cite[Proposition 3.26]{N Cat}, for any $\bn \in \nn$ and all large enough $N \gg M$, we have
\begin{equation}
\label{eqn:contain 2}
\CS_{\geq N \bone|\bn} \subseteq \oCS_{\geq M \bone |\bn}
\end{equation}
Due to the two displays above, an element $F \in \CS_{<\b0|-\bn}$ lies in $J^{\neq 0}(\bpsi)_{\bn}$ if and only if for all orderings $i_1,\dots,i_n$ of $\bn$ and all $d_1,\dots,d_n \geq M$ with $M$ large enough, we have
\begin{multline}
\label{eqn:multline}
0 = \left \langle e_{i_1,d_1} * \dots * e_{i_n,d_n} \prod_{i \in I} \prod_{a=1}^{n_i} \psi_i(z_{ia}), S(F) \right \rangle = \\ = (-1)^n \int_{1 \ll |z_1| \ll \dots \ll |z_n|} \frac {z_1^{d_1} \dots z_n^{d_n} F (z_1,\dots,z_n)}{\prod_{1\leq a < b \leq n} \zeta_{i_bi_a} \left(\frac {z_b}{z_a} \right)} \prod_{a=1}^n \psi_{i_a}(z_a) 
\end{multline}
The reason for having the variables run over circles of radius $\gg 1$ is so that any finite poles of $\psi$ should be contained inside these circles; this corresponds to the fact that the pairing above is computed by expanding each $\psi_i(z)$ as a power series in $z^{-1}$. If $M$ is much larger than the order of the poles of all $\psi_i(z)$, then the condition above was shown in \cite[proof of Theorem 1.3]{N Cat} to be equivalent to
\begin{equation}
\label{eqn:residue condition}
\underset{z_n=x_n}{\text{Res}} \dots \underset{z_1=x_1}{\text{Res}} \frac {z_1^{d_1} \dots z_n^{d_n} F(z_1,\dots,z_n)}{\prod_{1\leq a < b \leq n} \zeta_{i_bi_a} \left(\frac {z_b}{z_a} \right)} \prod_{a=1}^n \frac {\psi_{i_a}(z_a)}{z_a} = 0, \qquad \forall d_1,\dots,d_n \in \BZ
\end{equation}
where $(x_1,\dots,x_n)$ denotes an ordering of $\bx$, i.e. a one-to-one correspondence $x_a \leftrightarrow x_{i_a\bullet_a}$ which matches the one-to-one correspondence $z_a \leftrightarrow z_{i_a\bullet_a}$ that is implicit in all formulas above, see \eqref{eqn:notation ordering}. Equation \eqref{eqn:residue condition} imposes finitely many linear conditions on the value and derivatives of $F$ at $(z_1,\dots,z_n) = (x_1,\dots,x_n)$ (derivatives may arise due to the fact that the poles giving rise to the residues \eqref{eqn:residue condition} may have order $> 1$). In order for these conditions to be non-vacuous, the rational function in \eqref{eqn:residue condition} must have a pole at $(z_1,\dots,z_n) = (x_1,\dots,x_n)$, and a necessary condition for this is for every $z_b$ to either be a non-zero pole of $\psi_{i_b}(z)$, or equal to $z_a q^{-d_{i_ai_b}}$ for some $a<b$. \end{proof}

\medskip 

\begin{remark} Because of \eqref{eqn:contain 1}-\eqref{eqn:contain 2}, the following map is an isomorphism
$$
\CS_{<\b0}^- \Big / J^{\neq 0}(\bpsi) \xrightarrow{\sim} \CS^- \Big/ \bar{J}^{\neq 0}(\bpsi)
$$
where $\bar{J}^{\neq 0}(\bpsi)$ denotes the set of $F \in \CS^-$ such that \eqref{eqn:j pairing new} holds, or equivalently \eqref{eqn:residue condition} holds. Indeed, injectivity is obvious, while for surjectivity we argue as follows: 
$$
\forall R \in \CS_{-\bn}, \ \exists N \gg 0 \text{ s.t. } R \prod_{i \in I} (z_{i1} \dots z_{in_i})^N \in \CS_{<\b0|-\bn}
$$
Then we note that
$$
R \prod \left( \prod_{i \in I} (z_{i1} \dots z_{in_i})^N  - \prod_{i \in I} (x_{i1} \dots x_{in_i})^N \right)^M \in \bar{J}^{\neq 0}(\bpsi)
$$
where the first product runs over the finitely many non-trivial residues which appear in \eqref{eqn:residue condition} (and $M$ is chosen large enough to kill all the poles that may produce these residues). The formula above implies that a non-zero constant times $R$ is equivalent to a linear combination of elements of $\CS^-_{<\b0}$ modulo $\bar{J}^{\neq 0}(\bpsi)$.

\end{remark}

\medskip

\subsection{The contribution of 0}
\label{sub:zero}

Having studied the case of $\bx \in (\BC^*)^{\bn}$ in the previous Subsection, let us now consider the case of $\bx = \b0_{\bn}$. 

\medskip

\begin{definition}
\label{def:zero}
	
For any $\br \in \zz$, define
\begin{equation}
\label{eqn:zero}
L^\br = \CS_{<\b0}^- \Big/ J^\br 
\end{equation}
where
$J^\br = \oplus_{\bn \in \nn}J^\br_{\bn}$ consists of those $F\in \CS_{<\b0|-\bn}$ such that
\begin{equation}
\label{eqn:j pairing polynomial}
\left \langle E(z_{i1},\dots,z_{in_i}) \prod_{i \in I} \prod_{a=1}^{n_i} z_{ia}^{-r_i}, S \left( F(z_{i1},\dots,z_{in_i}) \right) \right \rangle = 0
\end{equation}
for all $E \in \CS_{\geq \b0|\bn}$.
	
\end{definition}

\medskip 

\noindent As opposed from $L(\bpsi)$, the vector space $L^{\br}$ carries an extra grading
\begin{equation}
	\label{eqn:grading simple polynomial}
\CS_{<\b0}^- \Big / J^{\br} = \bigoplus_{\bn \in \nn} \bigoplus_{d \in \BN} \CS_{<\b0|-\bn,d} \Big / J^{\br}_{\bn,d}
\end{equation}
which is induced by the vertical degree on $\CS_{<\b0}^-$ (with respect to which $J^\br$ is homogeneous). Moreover, $L^\br$ is a vertically graded $\BC[\ph_{i,d}^+]_{i \in I, d\geq 0}$-module via \eqref{eqn:cartan series}.

\medskip

\noindent For any rational $\ell$-weight $\bpsi$, let us write
$$
\ord \bpsi \in \zz
$$
for the $I$-tuple of orders of the poles of the rational functions $\psi_i(z)$ at $z=0$ (by assumption, the orders of these rational functions at $z = \infty$ are all 0). 

\medskip

\begin{proposition}
\label{prop:decomposition simple}

Let us write the coproduct \eqref{eqn:coproduct shuffle minus} as
\begin{equation}
\label{eqn:f'}
\Delta(F) = \sum_{\b0 \leq \bm \leq \bn} F'(z_{i1},\dots,z_{im_i})_{i \in I} \otimes F''(z_{i,m_i+1},\dots,z_{in_i})_{i \in I} \prod_{i \in I} \prod_{a=1}^{m_i} \ph_i^-(z_{ia})
\end{equation} 
For any $\ell$-weight $\bpsi$, the assignment $F \mapsto F' \otimes F''$ induces an isomorphism 
\begin{equation}
\label{eqn:decomposition simple}
L(\bpsi) \stackrel{\sim}\longrightarrow L^{\eord \bpsi} \otimes L^{\neq 0}(\bpsi), \qquad 
\end{equation}
of $\BC[\ph_{i,d}^+]_{i \in I, d\geq 0}$-modules, where $\ph_i^+(z)$ acts on the RHS by $\ph_i^+(z) \otimes \ph_i^+(z)$, $\forall i \in I$.

\end{proposition}

\medskip

\begin{proof} Let $\br = \ord \bpsi$. First of all, we note that although the implicit sum in $F' \otimes F''$ is infinite due to the power series expansion in \eqref{eqn:coproduct shuffle minus}, all but finitely many summands vanish in $L^{\br} \otimes L^{\neq 0}(\bpsi)$; this is because the	homogeneous degree of $F'$ can only increase in this power series expansion, and shuffle elements of sufficiently large degree lie in $J^\br$ (and thus vanish in $L^\br$) automatically. In order to construct the map \eqref{eqn:decomposition simple} and to show that it is injective, we must prove that 	
\begin{equation}
\label{eqn:that}
F \in J(\bpsi) \quad \Leftrightarrow \quad F' \otimes F'' \in J^{\br} \otimes \CS^- + \CS^- \otimes J^{\neq 0}(\bpsi)
\end{equation}
For the implication $\Rightarrow$, we must show that
\begin{multline}
\label{eqn:this}
\left \langle E'(z_{i1},\dots,z_{im_i})_{i \in I} \prod_{i \in I} \prod_{a=1}^{m_i} z_{ia}^{-r_i}, S(F') \right \rangle \\ \left \langle E''(z_{i,m_i+1},\dots,z_{in_i})_{i \in I} \prod_{i \in I} \prod_{a=m_i+1}^{n_i} \psi_i(z_{ia}), S(F'') \right \rangle = 0
\end{multline}
for all $E' \in \CS_{\geq \b0|\bm}$ and $E'' \in \CS_{\geq N \bone|\bn-\bm}$ with $N$ large enough. By \eqref{eqn:bialgebra 2} and the anti-automorphism property of the antipode, the condition above is equivalent to \footnote{Note that we must use the following straightforward consequence of \eqref{eqn:coproduct shuffle plus} and \eqref{eqn:bialgebra 1}:
\begin{equation}
\label{eqn:consequence 1}
\langle E'', S(\ph F'') \rangle = \varepsilon(\ph) \langle E'',S(F'')\rangle
\end{equation}
for any $E'' \in \CS^+$, $F'' \in \CS^-$ and $\ph = \ph^-_{i_1,d_1} \ph^-_{i_2,d_2}\dots$. We also note that $\ph$ is on opposite sides of $F''$ in the RHS of \eqref{eqn:f'} and in the LHS of \eqref{eqn:consequence 1}. This is not an issue, since commuting $\ph$ past $F''$ happens at the cost of multiplying $F' \otimes F''$ by a power series in $\{z_{ia}/z_{jb}\}_{i,j \in I, a \leq m_i, b > m_j}$ with non-zero constant term, which does not change the property that $F' \otimes F'' \in J^{\br} \otimes \CS^- + \CS^- \otimes J^{\neq 0}(\bpsi)$.}
$$
\left \langle E'(z_{i1},\dots,z_{im_i}) * E''(z_{i,m_i+1},\dots,z_{in_i}) \prod_{i \in I} \left( \prod_{a=1}^{m_i} z_{ia}^{-r_i} \prod_{a=m_i+1}^{n_i} \psi_i(z_{ia}) \right), S(F)\right \rangle = 0
$$
However, consider the following fact: since $\psi_i(w) = O(w^{-r_i})$ near 0, then we can find a polynomial $Q_i(w)$ such that $Q_i(w)\psi_i(w) \in w^{-r_i} + w^N\BC[w]$ for arbitrarily large $N$ (which we choose large enough so that any shuffle element of vertical degree $\geq N$ is automatically in $J^\br$). Thus, the equation above is implied by
$$
\left \langle E'(z_{i1},\dots,z_{im_i})\prod_{i \in I} \prod_{a=1}^{m_i} Q_i(z_{ia}) * E''(z_{i,m_i+1},\dots,z_{in_i}) \prod_{i \in I} \prod_{a=1}^{n_i} \psi_i(z_{ia}), S(F)\right \rangle = 0
$$
which in turn holds because $F \in J(\bpsi)$ and $E' \prod_{i,a} Q_i(z_{ia}) * E'' \in \CS^+_{\geq \b0}$.

\medskip

\noindent Let us now prove the implication $\Leftarrow$ of \eqref{eqn:that}. To do so, we will relate the vanishing condition \eqref{eqn:j pairing} with the vanishing conditions \eqref{eqn:j pairing new} and \eqref{eqn:j pairing polynomial} by explicitly working out the pairings involved. By \eqref{eqn:spherical}, any $E \in \CS_{\geq \b0|\bn}$ can be written (non-uniquely) as a linear combination of shuffle elements of the form
$$
E = \text{Sym} \left[ \nu(z_1,\dots,z_n) \prod_{1 \leq a < b \leq n} \zeta_{i_ai_b} \left(\frac {z_a}{z_b} \right) \right]
$$
for some Laurent polynomial $\nu$ and some ordering $i_1,\dots,i_n$ of $\bn$. Implicitly, we identify each symbol $z_a$ with $z_{i_a\bullet_a}$ for some $\bullet_a \in \BN$ (the choice of which does not matter due to the symmetrization in the formula above), see \eqref{eqn:notation ordering}. Then formula \eqref{eqn:antipode pairing shuffle} states that $F \in J(\bpsi)_{\bn}$ if and only if for any $E$ as above we have
\begin{equation}
\label{eqn:integral 1}
\int_{1 \ll |z_1| \ll \dots \ll |z_n|} \frac {\nu(z_1,\dots,z_n)F(z_1,\dots,z_n)}{ \prod_{1 \leq a < b \leq n} \zeta_{i_bi_a} \left(\frac {z_b}{z_a} \right)} \prod_{a=1}^n \psi_{i_a}(z_a) = 0
\end{equation}
We will move the contours of integration in the formula above past the non-zero poles of the rational functions $\psi_i(z)$ toward 0, which means that \eqref{eqn:integral 1} is equivalent to
\begin{multline}
\label{eqn:integral 2}
\sum_{\{1,\dots,n\} = \{a_1<\dots <a_k\} \sqcup \{b_1 < \dots < b_{\ell}\}} \underset{z_{b_\ell} \neq 0}{\text{Res}} \dots \underset{z_{b_1}\neq 0}{\text{Res}} \ \underset{z_{a_k}=0}{\text{Res}} \dots \underset{z_{a_1}=0}{\text{Res}} \\ \frac {\nu(z_1,\dots,z_n)F(z_1,\dots,z_n)}{ \prod_{1 \leq a < b \leq n} \zeta_{i_bi_a} \left(\frac {z_b}{z_a} \right)} \prod_{a=1}^n \frac {\psi_{i_a}(z_a)}{z_a} = 0
\end{multline}
where for any rational function $f$ we write
$$
\underset{w \neq 0}{\text{Res}} \ f(w) = \sum_{c \in \BC^*} \underset{w = c}{\text{Res}} \ f(w)
$$
In what follows, we will expand the rational functions on the second line of \eqref{eqn:integral 2} as $|z_{a_1}|,\dots,|z_{a_k}| \ll |z_{b_1}|,\dots,|z_{b_{\ell}}|$ and indicate this as $F(z_{a_1},\dots,z_{a_k} \otimes z_{b_1},\dots,z_{b_{\ell}})$ etc, as in \eqref{eqn:coproduct shuffle minus}. We conclude that $F \in J(\bpsi)_{\bn}$ if and only if
\begin{equation}
\label{eqn:integral 3}
\sum_{\{1,\dots,n\} = \{a_1<\dots <a_k\} \sqcup \{b_1 < \dots < b_{\ell}\}} \underset{z_{b_\ell} \neq 0}{\text{Res}} \dots \underset{z_{b_1}\neq 0}{\text{Res}} \ \underset{z_{a_k}=0}{\text{Res}} \dots \underset{z_{a_1}=0}{\text{Res}} 
\end{equation}
$$ 
\frac {\nu(z_{a_1},\dots,z_{a_k} \otimes z_{b_1},\dots,z_{b_{\ell}})F(z_{a_1},\dots,z_{a_k} \otimes z_{b_1},\dots,z_{b_{\ell}})}{ \prod_{1 \leq a < b \leq n} \zeta_{i_bi_a} \left(\frac {z_b}{z_a} \right)} \prod_{s=1}^{k} z_{a_s}^{-r_{i_{a_s}}-1} \prod_{t=1}^{\ell} \frac {\psi_{i_{b_t}}(z_{b_t})}{z_{b_t}} = 0
$$
Above, we used the fact that $\psi_i(w) =O(w^{-r_i})$ near 0 to replace $\psi_{i_{a_s}}(z_{a_s})$ by $z_{a_s}^{-r_{i_{a_s}}}$. This is allowed because $\CS_{\geq \b0}^+$ is closed under multiplication by color-symmetric polynomials, and thus \eqref{eqn:integral 2} holds for $\nu$ multiplied by any color-symmetric polynomial in $z_1,\dots,z_n$. Moreover, if the expression in \eqref{eqn:integral 3} has total degree in $z_{a_1},\dots,z_{a_k}$ large enough, the value of the residue in $z_{a_1},\dots,z_{a_k}$ at 0 will vanish; this also explains why only finitely many terms of the expansion $|z_{a_1}|,\dots,|z_{a_k}| \ll |z_{b_1}|,\dots,|z_{b_{\ell}}|$ survive the residue \eqref{eqn:integral 3}. Using relations \eqref{eqn:shuffle plus commute} and \eqref{eqn:coproduct shuffle plus}, it is straightforward to obtain
$$
\Delta(E) = \sum \ph E' \otimes E''
$$
where (above and henceforth) the symbol $\sum$ stands for summing over partitions $\{1,\dots,n\} = \{a_1<\dots<a_k\} \sqcup \{b_1 < \dots < b_\ell\}$ for various $k$, $\ell$, and we write
\begin{align*}
&\ph = \ph_{i_{b_1}}^+(z_{b_1}) \dots \ph_{i_{b_\ell}}^+(z_{b_\ell}) \\
&E' = \text{Sym} \left[ \nu'(z_{a_1},\dots,z_{a_k}) \prod_{1 \leq s < t \leq k} \zeta_{i_{a_s} i_{a_t}} \left(\frac {z_{a_s}}{z_{a_t}} \right) \right] \prod_{1 \leq s \leq k, 1 \leq t \leq \ell}^{a_s < b_t} \frac {\zeta_{i_{a_s}i_{b_t}} \left(\frac {z_{a_s}}{z_{b_t}} \right)}{\zeta_{i_{b_t}i_{a_s}} \left(\frac {z_{b_t}}{z_{a_s}} \right)} \\ 
&E'' = \text{Sym} \left[ \nu''(z_{b_1},\dots,z_{b_\ell}) \prod_{1 \leq s < t \leq \ell} \zeta_{i_{b_s} i_{b_t}} \left(\frac {z_{b_s}}{z_{b_t}} \right) \right] 
\end{align*} 
where $\nu(z_1,\dots,z_n) = \nu'(z_{a_1},\dots,z_{a_k}) \otimes \nu''(z_{b_1},\dots,z_{b_{\ell}})$.  Therefore, condition \eqref{eqn:integral 3} can be translated into the fact that $F \in J(\bpsi)_{\bn}$ if and only if \footnote{Note that we must use the following straightforward consequence of \eqref{eqn:coproduct shuffle minus} and \eqref{eqn:bialgebra 2}:
\begin{equation}
\label{eqn:consequence 2}
\langle \ph E', S(F') \rangle = \varepsilon(\ph) \langle E',S(F')\rangle
\end{equation}
for any $E' \in \CS^+$, $F' \in \CS^-$ and $\ph = \ph^+_{i_1,d_1} \ph^+_{i_2,d_2} \dots$.}
\begin{multline}
\label{eqn:integral 4}
\sum \left \langle E' \prod_{s=1}^{k} z_{a_s}^{-r_{i_{a_s}}} ,  S(F')  \right \rangle \\
\underset{z_{b_\ell} \neq 0}{\text{Res}} \dots \underset{z_{b_1}\neq 0}{\text{Res}}  \frac {\nu''(z_{b_1},\dots,z_{b_{\ell}})F''(z_{b_1},\dots,z_{b_{\ell}})}{ \prod_{1 \leq s<t \leq \ell} \zeta_{i_{b_t}i_{b_s}} \left(\frac {z_{b_t}}{z_{b_s}} \right)}  \prod_{t=1}^{\ell} \frac {\psi_{i_{b_t}}(z_{b_t})}{z_{b_t}} = 0
\end{multline}
The fact that $E \in \CS_{\geq \b0}^+ \Rightarrow E' \in \CS_{\geq \b0}^+$ (which is essentially a reformulation of \eqref{eqn:naive 1}) immediately allows us to conclude the $\Leftarrow$ implication of \eqref{eqn:that}: if $F' \in J^{\br}$ then the first line of \eqref{eqn:integral 4} vanishes, while if $F'' \in J^{\neq 0}(\bpsi)$ then the second line of \eqref{eqn:integral 4} vanishes due to \eqref{eqn:residue condition} being true for all $x_1,\dots,x_n \in \BC^*$. We have thus shown that the map \eqref{eqn:decomposition simple} is well-defined and injective. As explained in the proof of \cite[Theorem 1.3]{N Cat}, this map is also surjective, which concludes the proof of Proposition \ref{prop:decomposition simple}. \end{proof}

\medskip 

\subsection{$q$-characters}
\label{sub:q-characters}

For any representation $\CA^\geq \curvearrowright V$ in category $\CO$, its $q$-character is defined (following \cite{FR,HJ}) as
\begin{equation}
\label{eqn:q-character definition}
\chi_q(V) = \sum_{\bpsi} \dim_{\BC}(V_\bpsi) [\bpsi]
\end{equation}
where $[\bpsi]$ are formal symbols associated to rational $\ell$-weights $\bpsi$, and $V_\bpsi$ is the generalized eigenspace of $V$ on which the Cartan series $\ph^+_i(z)$ acts by $\psi_i(z)$, for all $i \in I$. In the particular case of a simple module, the description of the action of Cartan elements in \eqref{eqn:cartan series} implies the following formula for the $q$-character (\cite{N Cat})
\begin{multline}
\label{eqn:q-character simple}
\chi_q(L(\bpsi)) = [\bpsi] \sum_{\bn \in \nn} \sum_{\bx = (x_{ia})_{i \in I, a \in \{1,\dots,n_i\}} \in \BC^{\bn}} \\ \dim_{\BC} \left(  \CS_{<\b0|-\bn} \Big / J(\bpsi)_{\bn} \right)_{\bx} \underbrace{\left[\left( \prod_{i \in I} \prod_{a=1}^{n_i} \frac {z - x_{ia}q^{d_{ij}}}{zq^{d_{ij}}-x_{ia}} \right)_{j \in I} \right]}_{\text{this }\ell\text{-weight is }\prod_{i,a} A_{i,x_{ia}}^{-1} \text{ as per \cite{FR}}}
\end{multline}
where we recall that elements of $\BC^{\bn}$ are $I$-tuples of unordered sets $(x_{i1},x_{i2},\dots)$. Above, we write $[\bpsi][\bpsi'] = [\bpsi\bpsi']$ with respect to component-wise multiplication of $I$-tuples of power series/rational functions. Because the decomposition \eqref{eqn:decomposition simple} respects the action of $\BC[\ph_{i,d}^+]_{i \in I, d \geq 0}$, then it also respects $q$-characters in the following sense
\begin{equation}
\label{eqn:decomposition simple characters}
\chi_q(L(\bpsi)) = \chi_q^{\neq 0}(L(\bpsi)) \cdot \chi^{\ord \bpsi}
\end{equation}
where

\medskip

\begin{itemize}[leftmargin=*]
	
	\item $\chi_q^{\neq 0}(L(\bpsi))$ is defined by formula \eqref{eqn:q-character definition} with $V = L^{\neq 0}(\bpsi)$; the proof of Propositions \ref{prop:residues} and \ref{prop:decomposition simple} shows that this quantity simply picks out the summands of \eqref{eqn:q-character simple} corresponding to $\bx \in (\BC^*)^{\bn}$ instead of $\bx \in \BC^{\bn}$, i.e.
$$
\chi^{\neq 0}_q(L(\bpsi)) = [\bpsi] \sum_{\bn \in \nn} \sum_{\bx \in (\BC^*)^{\bn}} \\ \dim_{\BC} \left(  \CS_{<\b0|-\bn} \Big / J(\bpsi)_{\bn} \right)_{\bx} \left[\left( \prod_{i \in I} \prod_{a=1}^{n_i} \frac {z - x_{ia}q^{d_{ij}}}{zq^{d_{ij}}-x_{ia}} \right)_{j \in I} \right] 
$$
	
	\item for any $\br \in \zz$, we define
\begin{equation}
	\label{eqn:character polynomial}
	\chi^{\br} = \sum_{\bn \in \nn} \dim_{\BC} \left(\CS_{<\b0|-\bn} \Big/ J^{\br}_{\bn} \right) q^{-\bn}
\end{equation}
where $q^{-\bn}$ is identified with the $\ell$-weight $[(q^{-(\bn,\bs^i)})_{i \in I}]$.
	
\end{itemize}

\medskip

\begin{remark} In the notation of \cite{FR}, we have 
\begin{equation}
\label{eqn:fm}
\left[ \left( \frac {z - xq^{d_{ij}}}{zq^{d_{ij}}-x} \right)_{j \in I} \right] = A_{i,x}^{-1}
\end{equation}
and so formula \eqref{eqn:q-character simple} states that $\dim_{\BC}\left(  \CS_{<\b0|-\bn} \Big / J(\bpsi)_{\bn} \right)_{\bx}$ is the coefficient of 
$$
\prod_{i\in I} \left( A_{i,x_{i1}} \dots A_{i,x_{in_i}} \right)^{-1}
$$
in the renormalized $q$-character $\frac {\chi_q(L(\bpsi))}{[\bpsi]}$. 

\end{remark}

\medskip

\subsection{Characters} 
\label{sub:characters}

In formula \eqref{eqn:decomposition simple characters}, the second term in the right-hand side is character \eqref{eqn:character polynomial}; we call it a ``character" instead of a ``$q$-character" because all the constituent $\ell$-weights are actually constant power series. Because of the vertical grading in \eqref{eqn:grading simple polynomial}, we may actually refine this character as follows
\begin{equation}
	\label{eqn:q-character polynomial}
\chi_{\text{ref}}^{\br} = \sum_{\bn \in \nn} \sum_{d = 0}^{\infty} \dim_{\BC} \left(\CS_{<\b0|-\bn,d} \Big/ J^{\br}_{\bn,d} \right) q^{-\bn} v^d
\end{equation}
We are now ready to prove our main result.

\medskip

\begin{proof} \emph{of Theorem \ref{thm:main}:} We will actually prove the following claim for any Kac-Moody Lie algebra $\fg$. If the conjectural formula \eqref{eqn:conj} holds, then we have the following formula for the refined character associated to any $\br \in \zz$
\begin{equation}
	\label{eqn:thm}
	\chi_{\text{ref}}^{\br} = \prod_{\bn \in \nn \backslash \b0} \prod_{d=1}^{\max(0,\br \cdot \bn)} \left( \frac 1{1-q^{-\bn}v^d} \right)^{a_{\fg,\bn}}
\end{equation}
(for $\fg$ of finite type, \eqref{eqn:conj} was proved in \cite{N Cat} with $a_{\fg,\bn}$ being 1 or 0 depending on whether $\bn$ is or is not a positive root). We will now use the results in Section \ref{sec:dimensions} to prove \eqref{eqn:thm}. To this end, recall the shift automorphism \eqref{eqn:shift}, with respect to which condition \eqref{eqn:j pairing polynomial} states the following criterion for any $F \in \CS_{<\b0}^-$
$$
F \in J^\br \quad \Leftrightarrow \quad \Big \langle \CS^+_{\geq - \br}, S(F) \Big \rangle = 0
$$
Thus, if we consider the restriction of pairing \eqref{eqn:pairing shuffle}, twisted by the antipode
\begin{equation}
\label{eqn:restricted pairing}
\CS^+_{\geq -\br} \otimes \CS^-_{<\b0} \xrightarrow{\langle \cdot, S(\cdot) \rangle} \BC 
\end{equation}
then $J^{\br}$ is the kernel of the pairing \eqref{eqn:restricted pairing}. Thus, $\chi^\br_{\text{ref}}$ is none other than
\begin{equation}
\label{eqn:last iso}
\left(\text{graded dimension of }\frac {\CS_{<\b0}^-}{\text{Ker \eqref{eqn:restricted pairing}}} \right) = \left(\text{graded dimension of } \frac {\CS_{\geq -\br}^+}{\text{Ker \eqref{eqn:restricted pairing}}} \right) 
\end{equation}
with the equality above being the elementary statement that if $A$ and $B$ are graded vector spaces with finite-dimensional components, endowed with a pairing $A \otimes B \rightarrow \BC$ that respects the grading, then the quotients of $A$ and $B$ by the respective kernels of the pairing have the same dimension. 

\medskip 

\noindent By applying the shift automorphism $\sigma_{\br}$ again (recall that the pairing and antipode are invariant under shifts by Proposition \ref{prop:shifts}), the graded dimension in \eqref{eqn:last iso} is the same as the LHS of \eqref{eqn:key} for $\bp_1 = -\binfty$ and $\bp_2 = \br$. The RHS of \eqref{eqn:key} is the same as the RHS of \eqref{eqn:thm} (up to the change of variables $d \mapsto \br \cdot \bn - d$, which is due to $\sigma_\br$ and the isomorphism \eqref{eqn:last iso}), which concludes our proof. \end{proof}

\end{document}